\documentclass[11pt]{article}
\usepackage[leqno]{amsmath}
\usepackage{dsfont}
\usepackage{mathrsfs}
\usepackage{amsmath,amssymb}
\usepackage{amsfonts}

\usepackage{amsthm}
\usepackage{graphicx}
\usepackage{subfigure}
\usepackage{xcolor}

\usepackage{bbm}
\usepackage{mathrsfs}
\usepackage{txfonts}
\usepackage{color}
\usepackage{pict2e}
\usepackage{float}
\usepackage{palatino,epsfig,latexsym}
\usepackage{epsf,xy,epic,amscd}
\usepackage{lineno}

\usepackage{bibentry}
\usepackage[colorlinks,linkcolor=blue,anchorcolor=blue,citecolor=blue,]{hyperref}

\theoremstyle{plain}
\newtheorem{theorem}{Theorem}[section]
\newtheorem{lemma}[theorem]{Lemma}

\newtheorem{proposition}[theorem]{Proposition}

\newtheorem{conjecture}[theorem]{Conjecture}

\newtheorem{Bounded Diameter Lemma}[theorem]{Bounded Diameter Lemma}
\theoremstyle{definition}
\newtheorem{definition}[theorem]{Definition}

\newtheorem{remark}[theorem]{Remark}

\newcommand{\Hmm}[1]{\leavevmode{\marginpar{\tiny%
$\hbox to 0mm{\hspace*{-0.5mm}$\leftarrow$\hss}%
\vcenter{\vrule depth 0.1mm height 0.1mm width \the\marginparwidth}%
\hbox to
0mm{\hss$\rightarrow$\hspace*{-0.5mm}}$\\\relax\raggedright #1}}}

\hfuzz=\maxdimen
\DeclareFixedFont{\Acknowledgment}{OT1}{cmr}{bx}{n}{14pt}
\begin{document}
\numberwithin{equation}{section}

\title{Producing circle patterns via configurations}
\author{Ze Zhou}
\date{}

\maketitle

\begin{abstract}
This paper studies circle patterns from the viewpoint of configurations. By using the topological degree theory, we extend the  Koebe-Andreev-Thurston Theorem to include circle patterns with obtuse exterior intersection angles. As a consequence, we obtain a generalization of Andreev's Theorem which allows obtuse dihedral angles.

\medskip
\noindent{\bf Mathematics Subject Classifications (2020):} 52C26, 51M10, 57M50.

\end{abstract}



\section{Introduction}
A circle pattern $\mathcal P$ on the Riemann sphere $\hat{\mathbb C}$ is a collection of closed disks on $\hat{\mathbb C}$. The contact graph $G_\mathcal P$ of such a pattern $\mathcal P$ is the graph having a vertex for each disk, and having an edge between the vertices $u$ and $w$ if the corresponding disks $D_u, D_w$ intersect each other. For each edge $e=[u,w]$ of $G_\mathcal P$, we have the exterior intersection angle $\Theta(e)\in[0,\pi)$. We refer the readers to Stephenson's monograph~\cite{Stephenson} for basic background on circle patterns.

Let us consider a natural question: Given a graph $G$ and a function $\Theta:E\to[0,\pi)$ defined on the edge set of $G$, is there a circle pattern whose contact graph is isomorphic to $G$ and whose exterior intersection angle function is $\Theta$? If so, to what extent is the circle pattern unique? The following theorem was proved by Marden-Rodin~\cite{Marden-Rodin} based on ideas of Thurston~\cite{Thurston}. For the special case of tangent patterns, the result was due to Koebe~\cite{Koebe}. We say a circle pattern $\mathcal P=\{D_v\}_{v\in V}$ on $\hat{\mathbb C}$ is \textbf{irreducible} if $\cup_{v\in A}D_v\subsetneq \hat{\mathbb C}$ for every proper subset $A\subsetneq V$.

\begin{theorem}[Marden-Rodin]\label{T-1-1}
Let $\mathcal T$ be a triangulation of the sphere and let $\Theta:E\to [0,\pi/2]$ be a function satisfying the conditions below:
\begin{itemize}
\item[$(i)$] If $e_1,e_2,e_3$ form a simple closed curve, then $\sum_{\mu=1}^3\Theta(e_\mu)<\pi$.
\item[$(ii)$] If $e_1,e_2,e_3,e_4$ form a simple closed curve, then $\sum_{\mu=1}^4\Theta(e_\mu)<2\pi$.
\end{itemize}
Then there exists an \textbf{irreducible}
circle pattern $\mathcal P$ on the Riemann sphere $\hat{\mathbb C}$ with contact graph isomorphic to the $1$-skeleton of $\mathcal T$ and exterior intersection angles given by $\Theta$. Furthermore, $\mathcal P$ is unique up to linear and anti-linear fractional maps.
\end{theorem}

Thurston~\cite[Chap. 13]{Thurston} observed that there was a direct connection between circle patterns and hyperbolic polyhedra. More precisely, in the Poincar\'{e} model of hyperbolic 3-space $\mathbb H^3$, we associate each closed disk $D_v$ on $\hat{\mathbb C}=\partial \mathbb H^3$ with the hyperbolic half-space which is identical to the closed convex hull of ideal points in $\hat{\mathbb C}\setminus D_v\subset\partial \mathbb H^3$. Under suitable conditions, the intersection of half-spaces of a circle pattern gives a hyperbolic polyhedron, with combinatorics dual to the contact graph and dihedral angles equal to the corresponding exterior intersection angles.

Andreev~\cite{Andreev} obtained the following result which later played a significant role in the proof of Thurston's Hyperbolization Theorem for Haken 3-manifolds~\cite{Otal}. See also the work of Roeder-Hubbard-Dunbar \cite{Roeder-Hubbard-Dunbar} for a more detailed proof which corrects an error in Andreev's original paper~\cite{Andreev}. Given an abstract polyhedron $P$, we call a simple closed curve $\Gamma$ formed of $k$ edges of the dual complex $P^\ast$ a \textbf{$k$-circuit}, and if all of the endpoints of the edges of $P$ intersected by $\Gamma$ are distinct, we call such a circuit a \textbf{prismatic $k$-circuit}.

\begin{theorem}[Andreev]\label{T-1-2}
Let $P$ be an abstract trivalent polyhedron with more than four faces. Assume that $\Theta:E\to(0,\pi/2]$ is a function satisfying the conditions below:
\begin{itemize}
\item[$(i)$] Whenever three distinct edges $e_1,e_2,e_3$  meet at a vertex,  then $\sum_{\mu=1}^3\Theta(e_\mu)>\pi$.
\item[$(ii)$] Whenever $\Gamma$ is a \textbf{prismatic 3-circuit} intersecting edges $e_1,e_2,e_3$, then $\sum_{\mu=1}^3\Theta(e_\mu)<\pi$.
\item[$(iii)$] Whenever $\Gamma$ is a \textbf{prismatic 4-circuit} intersecting edges $e_1,e_2,e_3,e_4$, then $\sum_{\mu=1}^4\Theta(e_\mu)<2\pi$.
\item[$(iv)$] Whenever there is a four sided face bounded by edges $e_1,e_2,e_3,e_4$, enumerated successively, with edges $e_{12},e_{23},e_{34},e_{41}$ entering the four vertices (edge $e_{ij}$ connects to the common end of $e_i$ and $e_j$), then
\[
\begin{aligned}
&\Theta(e_{1})+\Theta(e_{3})
+\Theta(e_{12})+\Theta(e_{23})+\Theta(e_{34})+\Theta(e_{41})\,<\,3\pi,
\quad\text{and}\\
&\Theta(e_{2})+\Theta(e_{4})+\Theta(e_{12})
+\Theta(e_{23})+\Theta(e_{34})+\Theta(e_{41})\,<\,3\pi.
\end{aligned}
\]
\end{itemize}
Then there exists a compact convex hyperbolic polyhedron $Q$ combinatorially equivalent to $P$ with dihedral angles given by $\Theta$. Furthermore, $Q$ is unique up to isometries of $\mathbb H^3$.
\end{theorem}

\begin{remark}
If $P$ is not the triangular prism, condition $(iv)$ is implied by conditions $(ii)$ and $(iii)$ (see~\cite[Proposition 1.5]{Roeder-Hubbard-Dunbar}). Moreover, the hyperbolic polyhedra in above theorem correspond to \textbf{irreducible} circle patterns (see Lemma~\ref{L-3-8}).
\end{remark}

Thurston~\cite[Theorem 13.1.1]{Thurston} studied existence and rigidity of circle patterns on surfaces of genus $g>0$ with given non-obtuse exterior intersection angles. An adaptation of Thurston's consequence to Riemann sphere was proved by Bowers-Stephenson~\cite{Bowers-Stephenson} using ideas analogous to the famous Uniformization Theorem. In some sense, the work of Bowers-Stephenson~\cite{Bowers-Stephenson} provided a unified version of Theorem~\ref{T-1-1} and Theorem~\ref{T-1-2}.

One may ask whether it is possible to relax the requirement of non-obtuse angles in these theorems (see, e.g.,~\cite{He}). Rivin-Hodgson~\cite{Rivin-Hodgson} made a breakthrough by describing  all compact convex hyperbolic polyhedra in terms of a generalized Gauss map. Nonetheless, a remaining problem of Rivin-Hodgson's work~\cite{Rivin-Hodgson} is that there is not any satisfactory way of determining the combinatorics. Rivin~\cite{Rivin} later characterized the subclass of ideal hyperbolic polyhedra with arbitrary dihedral angles. After that, similar results were obtained by Bao-Bonahon~\cite{Bao-Bonahon} for hyperideal polyhedra. What is more, there are parallel results in the works of Bobenko-Springborn~\cite{Bobenko-Springborn} and Schlenker~\cite{Schlenker} for ideal and hyperideal circle patterns, respectively. Beyond the above settings, Zhou~\cite{Zhou}, Ge-Hua-Zhou~\cite{Ge-Hua-Zhou} and Jiang-Luo-Zhou~\cite{Jiang-Luo-Zhou} recently derived some direct generalizations of Marden-Rodin Theorem. Although allowing obtuse angles, these results~\cite{Zhou,Ge-Hua-Zhou,Jiang-Luo-Zhou} can not rule out the possibility that some combinatorially non-adjacent disks overlap.

A simple arc $\Gamma$ formed by edges of a triangulation $\mathcal T$ is said to be \textbf{homologically adjacent} if there is an edge between the starting and ending points of $\Gamma$, otherwise $\Gamma$ is said to be \textbf{homologically non-adjacent}. Below is our main result.

\begin{theorem}\label{T-1-4}
Let $\mathcal T$ be a triangulation of the sphere with more than four vertices. Assume that $\Theta:E\to [0,\pi)$ is a function satisfying the conditions below:
\begin{itemize}
\item[$\mathrm{\mathbf{(c1)}}$] If $e_1,e_2,e_3$ form the boundary of a triangle of $\mathcal T$, then $ \Theta(e_1)+\Theta(e_2)<\Theta(e_3)+\pi$, $\Theta(e_2)+\Theta(e_3)<\Theta(e_1)+\pi$, $\Theta(e_3)+\Theta(e_1)<\Theta(e_2)+\pi$.
\item[$\mathrm{\mathbf{(c2)}}$] If $e_1,e_2$ form a   \textbf{homologically non-adjacent} arc, then $\Theta(e_1)+\Theta(e_2)\leq\pi$, and one of the inequalities is strict when $\mathcal T$ is the boundary of a triangular bipyramid.
\item[$\mathrm{\mathbf{(c3)}}$] If $e_1,e_2,e_3$ form a simple closed curve separating the vertices of $\mathcal T$, then $\sum_{\mu=1}^3\Theta(e_\mu)<\pi$.
\item[$\mathrm{\mathbf{(c4)}}$] If $e_1,e_2,e_3,e_4$ form a simple closed curve separating the vertices of $\mathcal T$, then $\sum_{\mu=1}^4\Theta(e_\mu)<2\pi$.
\end{itemize}
Then there exists an \textbf{irreducible} circle pattern $\mathcal P$ on the Riemann sphere $\hat{\mathbb C}$ with contact graph isomorphic to the $1$-skeleton of $\mathcal T$ and exterior intersection angles given by $\Theta$.
\end{theorem}

\begin{remark}
Condition $\mathrm{\mathbf{(c1)}}$ was introduced by Zhou~\cite{Zhou} in view of the fact that it is satisfied when $\Theta(e_1),\Theta(e_2),\Theta(e_3)$ form the inner angles of a spherical triangle. Condition $\mathrm{\mathbf{(c2)}}$ is motivated by an angle relation (see Lemma~\ref{L-2-6}) for some three-circle configurations and is mainly applied to controlling the combinatorial type of the contact graph. In addition, we point out that conditions $\mathrm{\mathbf{(c1)}}, \mathrm{\mathbf{(c2)}}$ exclude some extreme cases (see, for example, the pattern in Figure~\ref{F-1}) in the work of Bowers-Stephenson~\cite{Bowers-Stephenson}. More details can be seen in Section~\ref{S-3}.
\end{remark}

\begin{figure}[htbp]\centering
\includegraphics[width=0.52\textwidth]{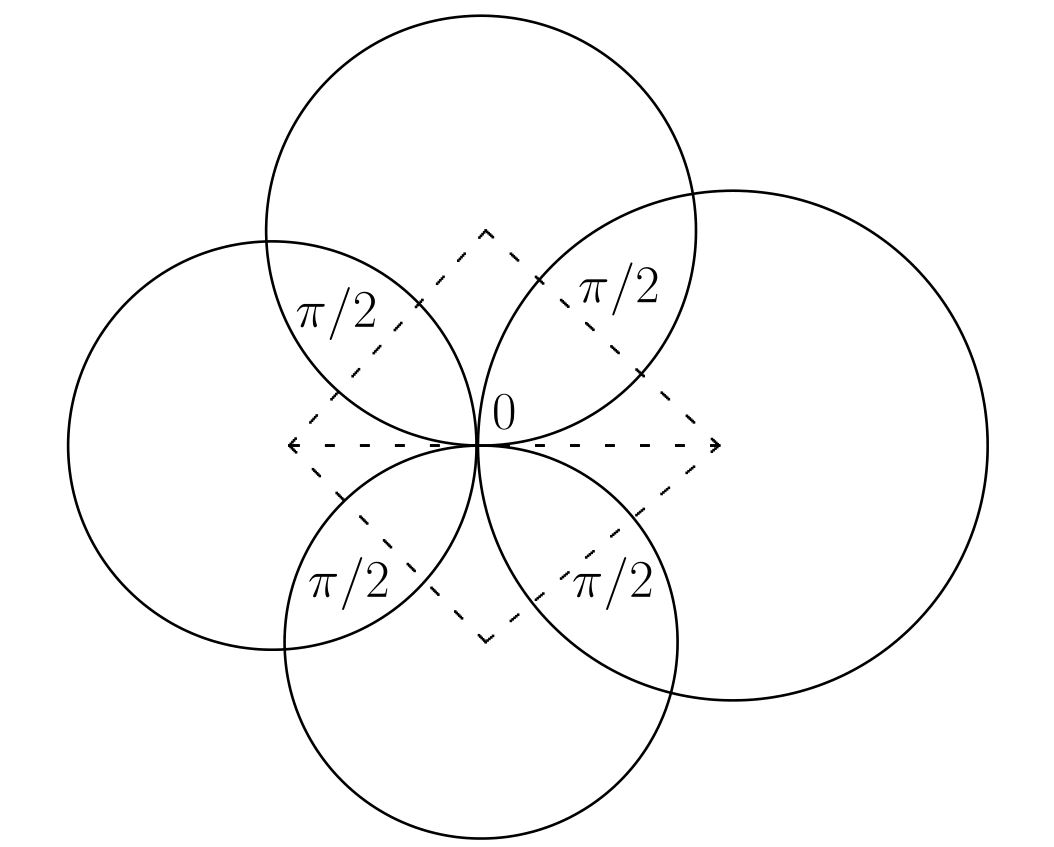}\label{F-1}
\caption{The exterior intersection function violates condition $\mathrm{\mathbf{(c1)}}$}
\end{figure}

Let $W$ be the set of functions $\Theta:E\to [0,\pi)$ satisfying the  conditions of Theorem~\ref{T-1-4}. Then $W$ is a convex subset of $[0,\pi)^{|E|}$. Sard's Theorem (Theorem~\ref{T-6-2}) provides an insight into local rigidity.
\begin{theorem}\label{T-1-6}
For almost every $\Theta\in W$, up to linear and anti-linear fractional maps, there are at most finitely many \textbf{irreducible} circle patterns on $\hat{\mathbb C}$ realizing the data $(\mathcal T,\Theta)$.
\end{theorem}

\begin{remark}\label{R-1-7}
We refer the readers to the works of Zhou~\cite{Zhou} and Bowers-Bowers-Pratt~\cite{Bowers-Bowers-Pratt} for some similar results. We also mention that Zhou~\cite{Zhou}, Xu~\cite{Xu} and Ge-Hua-Zhou~\cite{Ge-Hua-Zhou} established several global rigidity results under the following condition:
\begin{itemize}
\item[$\mathrm{\mathbf{(r1)}}$] If  $e_1,e_2,e_3$ form the boundary of a triangle of $\mathcal T$, then $ I(e_1)+I(e_2)I(e_3)\geq0$, $I(e_2)+I(e_3)I(e_1)\geq0$, $I(e_3)+I(e_1)I(e_2)\geq0$, where $I(e_\mu)=\cos\Theta(e_\mu)$ for $\mu=1,2,3$.
\end{itemize}
\end{remark}

Combining Theorem~\ref{T-1-4} and a rigidity result of Rivin-Hodgson \cite[Corollary 4.6]{Rivin-Hodgson}, we obtain the following generalization of Andreev's Theorem.

\begin{theorem}\label{T-1-8}
Let $P$ be an abstract trivalent polyhedron with more than four faces. Assume that $\Theta:E\to (0,\pi)$ is a function satisfying the  conditions below:
\begin{itemize}
\item[$\mathrm{\mathbf{(s1)}}$] Whenever three distinct edges $e_1,e_2,e_3$  meet at a vertex, then $\sum_{\mu=1}^3\Theta(e_\mu)>\pi$, and
    $\Theta(e_1)+\Theta(e_2)<\Theta(e_3)+\pi$, $\Theta(e_2)+\Theta(e_3)<\Theta(e_1)+\pi$, $\Theta(e_3)+\Theta(e_1)<\Theta(e_2)+\pi$.
\item[$\mathrm{\mathbf{(s2)}}$]Whenever $\Gamma$ is a \textbf{homologically non-adjacent} arc intersecting edges $e_1,e_2$, then $\Theta(e_1)+\Theta(e_2)\leq \pi$, and one of the inequalities is strict if $P$ is the triangular prism.
\item[$\mathrm{\mathbf{(s3)}}$] Whenever $\Gamma$ is a \textbf{prismatic 3-circuit} intersecting edges $e_1,e_2,e_3$, then $\sum_{\mu=1}^3\Theta(e_\mu)<\pi$.
\item[$\mathrm{\mathbf{(s4)}}$] Whenever $\Gamma$ is a \textbf{prismatic 4-circuit} intersecting edges $e_1,e_2,e_3,e_4$, then $\sum_{\mu=1}^4\Theta(e_\mu)<2\pi$.
\end{itemize}
Then there exists a compact convex hyperbolic polyhedron $Q$ combinatorially equivalent to $P$ with dihedral angles given by $\Theta$. Furthermore, $Q$ is unique up to isometries of $\mathbb H^3$.
\end{theorem}

\begin{figure}[htbp]\centering
\includegraphics[width=0.54\textwidth]{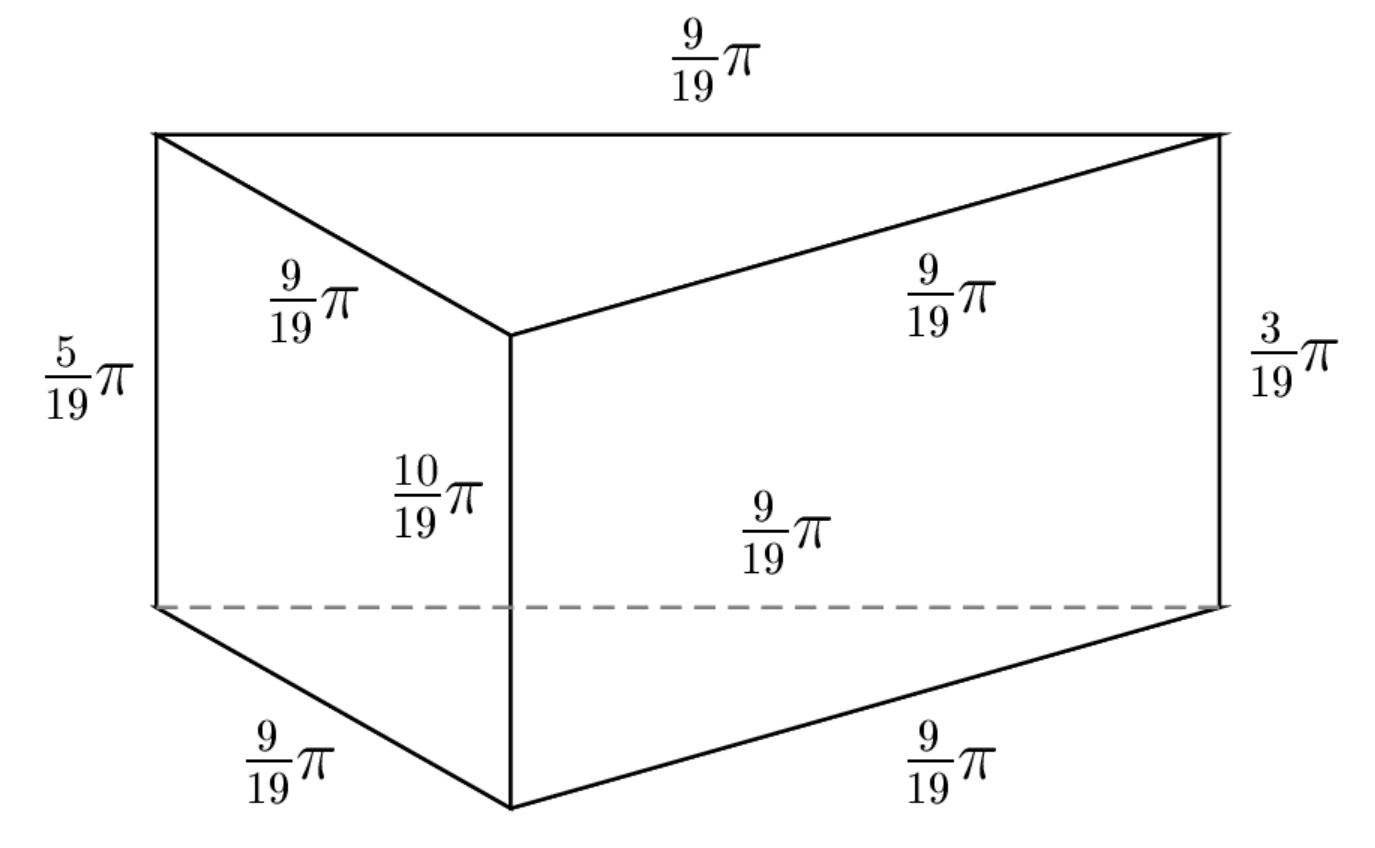}
\caption{A triangular prism having an obtuse dihedral angle}
\end{figure}

\begin{remark}
A more challenging problem is to describe the possible dihedral angles of compact convex hyperbolic polyhedra in a fixed combinatorial class. D\'{\i}az~\cite{Diaz1,Diaz2} and Roeder~\cite{Roeder} illustrate that some non-linear conditions are necessary to avoid those $m$-sided ($m\geq4$) faces degenerating to points or line segments. Under condition of non-obtuse angles, the Gauss-Bonnet Formula provides a useful tool to prevent such degeneracies~\cite{Andreev,Roeder-Hubbard-Dunbar}. Our observation (Lemma~\ref{L-2-6}) indicates that condition $\mathrm{\mathbf{(s2)}}$ is approximately adequate for this purpose as well.
\end{remark}

It is timely to give an outline of the proof of Theorem~\ref{T-1-4}. Our main tools are configurations spaces and topological degree theory. A key point is to deduce that under appropriate conditions a sequence of irreducible circle patterns $\{\mathcal P_n\}$ can produce a limit circle pattern $\mathcal P_\infty$ (see Lemma~\ref{L-3-4}). To attack the problem, we first show $\mathcal P_\infty$ exists in some reasonable compactification and then prove $\mathcal P_\infty$ satisfies the required properties step by step. For instance, using Lemma~\ref{L-2-1} and Lemma~\ref{L-2-5}, we assert that no disk in $\mathcal P_\infty$ degenerates to a point (see Proposition~\ref{P-3-5}); By Lemma~\ref{L-2-6} and Lemma~\ref{L-3-2}, we derive that any two combinatorially non-adjacent disks in $\mathcal P_\infty$ are disjoint. Surprisingly, this approach is parallel with a well-known method in PDE theory, where $\mathcal P_\infty$ plays a similar role to a weak solution, some lemmas in next section play similar roles to prior estimates, and configuration spaces play similar roles to Sobolev spaces.

The paper is organized as follows. In next section, we establish some preliminary results. In Section~\ref{S-3}, we study circle patterns from the viewpoint of spherical geometry and prove
Theorem~\ref{T-1-4} and Theorem~\ref{T-1-6} under an extra assumption. As a consequence, Theorem~\ref{T-1-8} is obtained. In Section~\ref{S-4}, we consider circle patterns from the viewpoint of Euclidean geometry and finish the proof of the main results. In Section~\ref{S-5}, we pose some questions for further developments. The last section contains an appendix regarding to some results from differential topology.



\section{Preliminaries}\label{S-2}
\subsection{A combinatorial fact}
We begin with the following elementary result which unveils some information behind the conditions of Theorem~\ref{T-1-4}.
\begin{lemma}\label{L-2-1}
Suppose $\mathcal T$ possesses more than four vertices. Under conditions $\mathrm{\mathbf{(c1)}},\mathrm{\mathbf{(c2)}},\mathrm{\mathbf{(c3)}},\mathrm{\mathbf{(c4)}}$, if $e_1,e_2,\cdots,e_k$ form a simple closed curve $\Gamma$ which is not the boundary of a triangle of $\mathcal T$, then $\sum_{\mu=1}^k\Theta(e_\mu)\leq(k-2)\pi$. Moreover, if $\Gamma$ is not the boundary of the union of two adjacent triangles, then the strict inequality holds.
\end{lemma}

\begin{proof}
When $k=3$, it is an immediate consequence of condition $\mathrm{\mathbf{(c3)}}$.

When $k\geq4$, we proceed by induction on $k$. For $k=4$, if $\Gamma$ is not the boundary of the union of two adjacent triangles, then $\Gamma$ separates the vertices of $\mathcal T$ and the statement follows from condition $\mathrm{\mathbf{(c4)}}$. If $\Gamma$ is the boundary of the union of two adjacent triangles, we may assume $e_1,e_2$ belong to a  triangle and $e_3,e_4$ belong to another one.

In case that $e_2,e_3$ form a \textbf{homologically non-adjacent} arc, then $e_1,e_4$ also form a \textbf{homologically non-adjacent} arc and the inequality is a result of condition $\mathrm{\mathbf{(c2)}}$. In case that $e_2,e_3$ form a \textbf{homologically adjacent} arc, there exists an edge $e_5$ such that $e_2,e_3,e_5$ form a simple closed curve $\Gamma_1$. Meanwhile, the edges $e_1,e_4,e_5$  form a simple closed curve $\Gamma_2$. Since $\mathcal T$ possesses more than four vertices, we may assume $\Gamma_1$ separates the vertices of $\mathcal T$ without loss of generality. Condition $\mathrm{\mathbf{(c3)}}$ gives
\[
\Theta(e_2)+\Theta(e_3)+\Theta(e_5)\,<\,\pi.
\]
Recall that $e_1,e_4,e_5$ also form a simple closed curve. Using conditions $\mathrm{\mathbf{(c1)}}, \mathrm{\mathbf{(c3)}}$, we get
\[
\Theta(e_1)+\Theta(e_4)\,<\,\pi+\Theta(e_5)\quad\text{or}\quad
\Theta(e_1)+\Theta(e_4)+\Theta(e_5)\,<\,\pi.
\]
No matter which case occurs, it follows that
\[
\sum\nolimits_{\mu=1}^4\Theta(e_\mu)\,<\,2\pi.
\]

Now assume that $k>4$ and the statement holds for all simple closed curves formed by $k-1$ edges. We divide the proof into the following situations:
\begin{itemize}
\item[$(i)$] For $\mu=1,2,\cdots,k$, the edges $e_\mu,e_{\mu+1}$ (set $e_{k+1}=e_1$) always form a \textbf{homologically non-adjacent} arc. Owing to condition $\mathrm{\mathbf{(c2)}}$, we deduce
    \[
    \sum\nolimits_{\mu=1}^k\Theta(e_\mu)\,\leq\,\dfrac{k}{2}\pi\,<\,(k-2)\pi.
    \]
\item[$(ii)$] There exists $\mu_0\in\{1,2,\cdots,k\}$ such that  $e_{\mu_0},e_{\mu_{0}+1}$ form a \textbf{homologically adjacent} arc. Thus we can find an edge $e_{\tau_0}$ such that  $e_{\mu_0},e_{\mu_{0}+1}, e_{\tau_0}$ form a simple closed curve. As before, we have
\[
\Theta(e_{\mu_0})+\Theta(e_{\mu_{0}+1})\,<\,\pi+\Theta(e_{\tau_0})\quad\mathrm{\text{or}}\quad
\Theta(e_{\mu_0})+\Theta(e_{\mu_0+1})+\Theta(e_{\tau_0})\,<\,\pi.
\]
Note that the $k-1$ edges $e_1,\cdots,e_{\mu_{0}-1},e_{\tau_0},e_{\mu_{0}+2},\cdots,e_k$  also form a simple closed curve. By induction,
\[
\sum\nolimits_{\mu=1}^{k}\Theta(e_\mu)-\Theta(e_{\mu_0})-\Theta(e_{\mu_{0}+1})
+\Theta(e_{\tau_0})\,\leq\,(k-3)\pi.
\]
In light of the above two relations, we prove
\[\sum\nolimits_{\mu=1}^{k}\Theta(e_\mu)\,<\,(k-2)\pi.
\]
\end{itemize}
\end{proof}

\subsection{Three-circle configurations}
Below we establish several lemmas on three-circle configurations. It should be pointed out that some special cases of theses results have appeared in the works of Thurston~\cite{Thurston}, Marden-Rodin \cite{Marden-Rodin}, Chow-Luo \cite{Chow-Luo}, and others and have played crucial roles in many proofs of Koebe-Andreev-Thurston Theorem.

\begin{lemma}\label{L-2-2}
Suppose $\Theta_i,\Theta_j,\Theta_k\in(0,\pi)$ are three angles satisfying
\[
\Theta_i+\Theta_j+\Theta_k\,>\,\pi,\;\;
\Theta_i+\Theta_j\,<\,\Theta_k+\pi,\;\;
\Theta_j+\Theta_k\,<\,\Theta_i+\pi,\;\;
\Theta_k+\Theta_i\,<\,\Theta_j+\pi.
\]
For any three numbers $r_i,r_j,r_k\in(0,\pi)$,  there exists a configuration of three intersecting disks in spherical geometry, unique up to isometries, having radii $r_i, r_j, r_k$  and meeting in exterior intersection angles $\Theta_i, \Theta_j, \Theta_k$.
\end{lemma}

\begin{figure}[htbp]\centering
\includegraphics[width=0.52\textwidth]{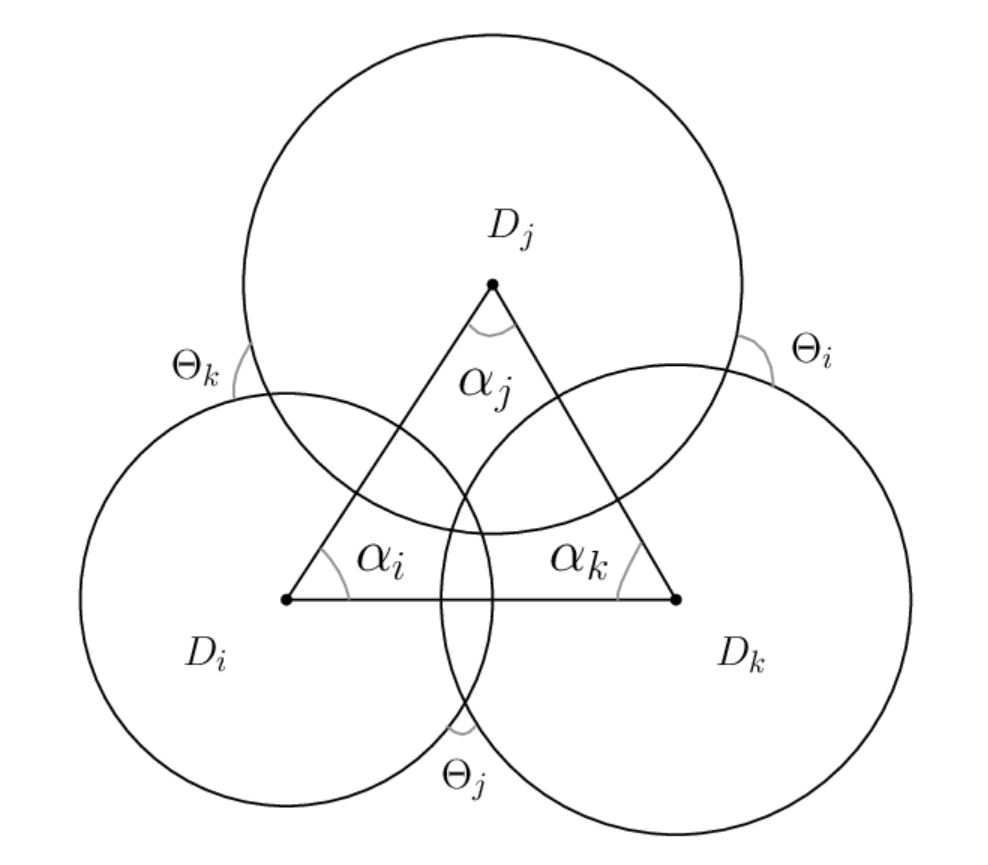}
\caption{A three-circle configuration}\label{F-3}
\end{figure}

\begin{proof}
Set
\[
l_i\, =\, \arccos(\cos r_{j}\cos r_{k}-\cos\Theta_{i}\sin r_{j}\sin r_{k})
\]
and $l_j,l_k$ similarly. We prove the lemma by showing the following  inequalities:
\[
l_i+l_j\,>\,l_k,\quad l_j+l_k\,>\,l_i,\quad l_k+l_i\,>\,l_j, \quad l_i+l_j+l_k\,<\,2\pi.
\]
Equivalently, let us verify
\[
\sin\dfrac{l_i+l_j+l_k}{2}\sin\dfrac{l_i+l_j-l_k}{2}\sin\dfrac{l_j+l_k-l_i}{2}\sin\dfrac{l_k+l_i-l_j}{2}>0.
\]
To simplify notations, for $\mu=i,j,k$, set $a_\mu=\cos r_\mu,$   $x_\mu=\sin r_\mu$. Then
\[
\cos l_i\,=\,a_{j}a_{k}-\cos\Theta_ix_{j}x_{k}.
\]
Notice that
\[
\begin{aligned}
&\,\sin\dfrac{l_i+l_j+l_k}{2}\sin\dfrac{l_i+l_j-l_k}{2}
\sin\dfrac{l_j+l_k-l_i}{2}\sin\dfrac{l_k+l_i-l_j}{2} \\
\,=&\,\dfrac{\sin^2 l_i\sin^2 l_j-(\cos l_i\cos l_j-\cos l_k)^2}{4}.
\end{aligned}
\]
Putting the above relations together, we need to demonstrate
\begin{equation}\label{E-2-1}
\begin{aligned}
&\sin^2\Theta_ix_j^2x^2_k+\sin^2\Theta_jx^2_k x_i^2+\sin^2\Theta_kx^2_ix^2_j-(2+2\cos\Theta_i\cos\Theta_j\cos\Theta_k)x_i^2x_j^2x_k^2\\
\,&\;\;+2\lambda_{ijk}a_ja_kx_jx_kx_i^2+2\lambda_{jki}a_ka_ix_kx_ix_j^2+2\lambda_{kij}a_ia_jx_ix_jx_k^2\,>\,0,
\end{aligned}
\end{equation}
where
\[
\lambda_{ijk}\,=\,\cos\Theta_i+\cos\Theta_j\cos\Theta_k.
\]
Since $a^2_\mu+x^2_\mu=1$, an equivalent form of~\eqref{E-2-1} is
\begin{equation}\label{E-2-2}
\begin{split}
&\sin^2\Theta_ia_i^2x_j^2x^2_k+\sin^2\Theta_ja_j^2x^2_k x_i^2+\sin^2\Theta_ka_k^2x^2_ix^2_j
+\zeta_{ijk}x_i^2x_j^2x_k^2\\
&\quad+2\lambda_{ijk}a_ja_kx_jx_kx_i^2+2\lambda_{jki}a_ka_ix_kx_ix_j^2+2\lambda_{kij}a_ia_jx_ix_jx_k^2\;>\;0,
\end{split}
\end{equation}
where
\[
\zeta_{ijk}\,=\,
\sin^2\Theta_i+\sin^2\Theta_j+\sin^2\Theta_k-(2+2\cos\Theta_i\cos\Theta_j\cos\Theta_k).
\]
Observe that $\Theta_i,\Theta_j,\Theta_k$ form the inner angles of
a spherical triangle. Let $\phi_i,\phi_j,\phi_k$ be the lengths of sides opposite to $\Theta_i,\Theta_j,\Theta_k$, respectively. The second spherical law of cosines gives
\begin{equation}\label{E-2-3}
\lambda_{ijk}\,=\,\cos\Theta_i+\cos\Theta_j\cos\Theta_k\,=\, \cos\phi_i\sin\Theta_j\sin\Theta_k.
\end{equation}
Set
\[
s_{i}\,=\,\sin\Theta_ia_ix_jx_k,\quad s_{j}\,=\,\sin\Theta_ja_jx_kx_i,\quad   s_{k}\,=\,\sin\Theta_ka_kx_ix_j.
\]
Inserting~\eqref{E-2-3} into~\eqref{E-2-2}, we reduce the proof to showing
\[
s_{i}^2+s_{j}^2+s_{k}^2+2\cos \phi_is_{j}s_{k}+2\cos \phi_j s_{k}s_{i}+2\cos \phi_ks_{i}s_{j}+\zeta_{ijk}x_i^2x_j^2x_k^2\,>\,0.
\]
For this purpose, completing the square gives
\[
\begin{aligned}
&\, s_{i}^2+s_{j}^2+s_{k}^2+2\cos \phi_is_{j}s_{k}+
2\cos \phi_j s_{k}s_{i}+2\cos \phi_ks_{i}s_{j}\\
=\,&\,(s_{i}+\cos \phi_js_{k}+\cos\phi_ks_{j})^2
+\sin^2\phi_js^2_{k}+\sin^2\phi_ks^2_{j}
+2(\cos\phi_i-\cos\phi_j\cos\phi_k)s_{j}s_{k}\\
\geq\, &\, \sin^2\phi_js^2_{k}+\sin^2\phi_ks^2_{j}+
2(\cos\phi_i-\cos\phi_j\cos\phi_k)s_{j}s_{k}.
\end{aligned}
\]
Using the spherical law of cosines, we obtain
\[
\cos\phi_i-\cos\phi_j\cos\phi_k\,=\,\cos\Theta_i\sin\phi_j\sin\phi_k.
\]
It follows that
\[
\begin{aligned}
&\, s_{i}^2+s_{j}^2+s_{k}^2+2\cos \phi_is_{j}s_{k}+
2\cos \phi_j s_{k}s_{i}+2\cos \phi_ks_{i}s_{j}\\
\geq &\,\sin^2\phi_js^2_{k}+\sin^2\phi_ks^2_{j}+
2\cos\Theta_i\sin\phi_j\sin\phi_ks_{j}s_{k}\\
=&\,(\sin\phi_js_{k}+\cos\Theta_i\sin\phi_ks_{j})^2
+\sin^2\Theta_i\sin^2\phi_ks^2_{j}\\
\geq &\,\sin^2\phi_k\sin^2\Theta_i\sin^2\Theta_ja_j^2x_k^2x_i^2,
\end{aligned}
\]
Meanwhile, a routine computation yields
\[
\begin{aligned}
\zeta_{ijk}\,&=\,\sin^2\Theta_i+\sin^2\Theta_j+\sin^2\Theta_k-(2+2\cos\Theta_i\cos\Theta_j\cos\Theta_k)\\
         \,&=\,\sin^2\Theta_i\sin^2\Theta_j-(\cos\Theta_k+\cos\Theta_i\cos\Theta_j)^2\\
        \,&=\,\sin^2\Theta_i\sin^2\Theta_j-\cos^2\phi_k\sin^2\Theta_i\sin^2\Theta_j\\
         \,&=\,\sin^2\phi_k\sin^2\Theta_i\sin^2\Theta_j.\\
\end{aligned}
\]
As a result,
\[
\begin{aligned}
\,&s_{i}^2+s_{j}^2+s_{k}^2+2\cos \phi_is_{j}s_{k}+2\cos \phi_j s_{k}s_{i}+2\cos \phi_ks_{i}s_{j}+\zeta_{ijk}x_i^2x_j^2x_k^2\\
\,\geq\,&\sin^2\phi_k\sin^2\Theta_i \sin^2\Theta_j(a_j^2x_k^2x_i^2+x_j^2x_k^2 x_i^2)\\
\,>\,&0.
\end{aligned}
\]
We thus finish the proof.
\end{proof}

\begin{lemma}\label{L-2-3}
For any $r_i,r_j,r_k,\Theta_i,\Theta_j,\Theta_k\in(0,\pi)$, define $l_i,l_j,l_k$  as above. If $\Theta_i+\Theta_j+\Theta_k=\pi$, then
\[
l_i+l_j\,>\,l_k,\quad l_j+l_k\,>\,l_i,\quad l_k+l_i\,>\,l_j, \quad l_i+l_j+l_k\,\leq\,2\pi.
\]
Moreover, when $r_i+r_j+r_k<\pi$, the last inequality is strict.
\end{lemma}

\begin{figure}[htbp]\centering
\includegraphics[width=0.46\textwidth]{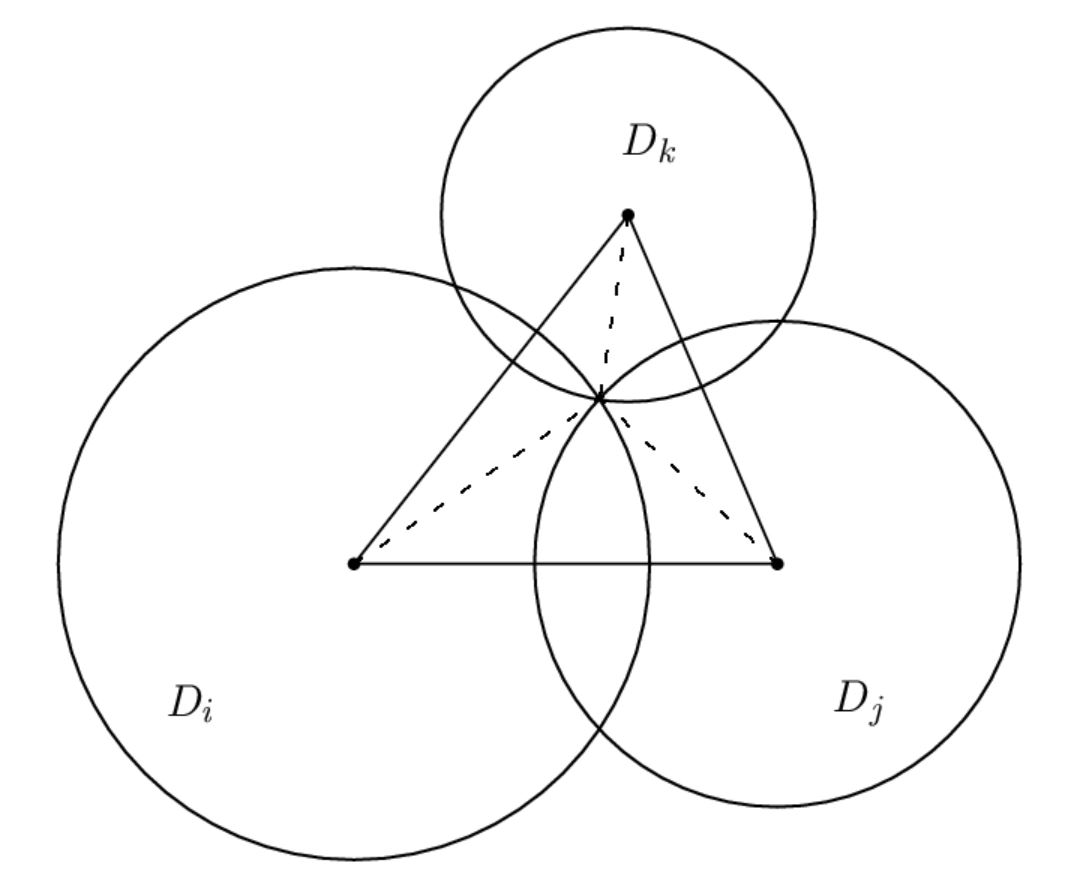}
\caption{A three-circle configuration via gluing process}\label{F-4}
\end{figure}

\begin{proof}
For  $r_i,r_j,\Theta_k\in(0,\pi)$, there are two intersecting disks $D_i,D_j$ in spherical geometry having radii $r_i,r_j$ and meeting in exterior intersection angle $\Theta_k$. As in Figure~\ref{F-4}, connecting the centers of $D_i,D_j$ and one of intersection points of their boundaries by geodesic segments, we obtain a spherical triangle $\triangle_{ij}$. Similarly, we  construct triangles $\triangle_{jk},\triangle_{ki}$.

Gluing $\triangle_{ij},\triangle_{jk},\triangle_{ki}$ along the corresponding edges produces a spherical triangle $\triangle_{ijk}$ having possibly a cone type singularity. Precisely, the cone angle is equal to
\[
\pi-\Theta_k+\pi-\Theta_i+\pi-\Theta_j\,=\,2\pi,
\]
which indicates the singularity actually vanishes.

Note that $\triangle_{ijk}$ has sides of lengths $l_i,l_j,l_k$. The triangle inequality of the spherical metric space then implies the former three inequalities of the lemma. Meanwhile, the last inequality follows from a limit argument
of Lemma~\ref{L-2-2}.

Finally, when $r_i+r_j+r_k<\pi$, it is easy to see
\[
l_i+l_j+l_k\,<\,2(r_i+r_j+r_k)\,<\,2\pi.
\]
\end{proof}

In Euclidean geometry, we have the following result.
\begin{lemma}\label{L-2-4}
Suppose $\Theta_i, \Theta_j, \Theta_k \in [0,\pi) $ are three angles satisfying
\[
\Theta_i+\Theta_j+\Theta_k\,\leq\,\pi
\]
or
\[
\Theta_i+\Theta_j\,<\,\Theta_k+\pi,\;\;
\Theta_j+\Theta_k\,<\,\Theta_i+\pi,\;\;
\Theta_k+\Theta_i\,<\,\Theta_j+\pi.
\]
For any three positive numbers $r_i,r_j,r_k$, there exists a configuration of three intersecting disks in Euclidean geometry, unique up to isometries, having radii $r_i,r_j,r_k$  and meeting in exterior intersection angles $\Theta_i,\Theta_j,\Theta_k$.
\end{lemma}

\begin{proof}
The lemma has actually appeared in the work of Jiang-Luo-Zhou~\cite{Jiang-Luo-Zhou}. Here we give an independent proof for the sake of completeness. Set
\[
l_i\,=\,\sqrt{r_j^2+r_k^2+2\cos\Theta_ir_jr_k}
\]
and $l_j,l_k$ similarly. It suffices to verify that $l_i,l_j,l_k$ satisfy the triangle inequalities. Without loss of generality, let us demonstrate
\[
|l_i-l_j|\,<\,l_k.
\]
Namely,
\[
\Big|\sqrt{r_j^2+r_k^2+2\cos\Theta_ir_jr_k}-\sqrt{r_k^2+r_i^2+2\cos\Theta_jr_kr_i} \,\Big|\,<\,\sqrt{r_i^2+r_j^2+2\cos\Theta_kr_ir_j}.
\]
Simplifying the above inequality, we need to show
\begin{equation}\label{E-2-4}
\sin^2\Theta_ir_j^2r_k^2+\sin^2\Theta_jr_k^2 r_i^2+\sin^2 \Theta_kr^2_ir^2_j
+2\lambda_{ijk}r_jr_k r_i^2+2\lambda_{jki}r_kr_i r_j^2+2\lambda_{kij} r_ir_j r_k^2\,>\,0.
\end{equation}

Now we divide the proof into the following situations:
\begin{itemize}
\item[$(i)$] $\Theta_i+\Theta_j+\Theta_k\leq \pi$. Then
\[
\begin{aligned}
\lambda_{ijk}\,&=\,\cos\Theta_i+\cos(\Theta_j+\Theta_k)+\sin\Theta_j\sin\Theta_k\\
&\geq\,2\cos\dfrac{\Theta_i+\Theta_j+\Theta_k}{2}\cos\dfrac{\Theta_i-\Theta_j-\Theta_k}{2}\\
&\geq\,0.
\end{aligned}
\]
Similarly,
\[
\lambda_{jki}\,\geq\, 0,\;\;\lambda_{kij}\,\geq\, 0.
\]
It is easy to see~\eqref{E-2-4} holds.

\item[$(ii)$] $\Theta_i+\Theta_j+\Theta_k >\pi$ and
    $\Theta_i+\Theta_j<\pi+\Theta_k,$ $\Theta_j+\Theta_k<\pi+\Theta_i,$ $\Theta_k+\Theta_i<\pi+\Theta_j$. As before, there exists a spherical triangle with inner angles $\Theta_i,\Theta_j,\Theta_k$. Suppose $\phi_i,\phi_j,\phi_k$ are the lengths of sides opposite to $\Theta_i,\Theta_j,\Theta_k$, respectively.
Set
\[
y_{i}\,=\,\sin\Theta_ir_jr_k,\quad y_{j}\,=\,\sin\Theta_jr_kr_i, \quad y_{k}\,=\,\sin\Theta_kr_ir_j.
\]
By~\eqref{E-2-3}, inequality~\eqref{E-2-4} is reduced to
\[
y_{i}^2+y_{j}^2+y_{k}^2+2\cos \phi_iy_{j}y_{k}+2\cos \phi_j y_{k}y_{i}+2\cos \phi_ky_{i}y_{j}\,>\,0.
\]
Similarly, completing the square yields
\[
\begin{aligned}
&y_{i}^2+y_{j}^2+y_{k}^2+2\cos \phi_iy_{j}y_{k}+
2\cos \phi_j y_{k}y_{i}+2\cos \phi_ky_{i}y_{j}\\
\,\geq\,&\sin^2\phi_k\sin^2\Theta_i\sin^2\Theta_jr_k^2r_i^2\\
\,>\,&0.
\end{aligned}
\]
\end{itemize}
In summary, the lemma is proved.
\end{proof}

As in Figure~\ref{F-3}, let $\triangle_{ijk}$ be the triangle whose vertices are the centers of disks $D_i,D_j,D_k$ and let $\alpha_i,\alpha_j,\alpha_k$ be the corresponding inner angles.

\begin{lemma}\label{L-2-5}
In both spherical and Euclidean geometries, we have
\begin{subequations}
\begin{align}
&\lim_{(r_i,r_j,r_k)\to(0,a,b)}\alpha_i
\,=\,\pi-\Theta_i,\label{E-2-5a}\\
&\lim_{(r_i,r_j,r_k)\to(0,0,c)}\alpha_i+\alpha_j
\,=\,\pi,\label{E-2-5b}\\
&\lim_{(r_i,r_j,r_k)\to(0,0,0)}\alpha_i+\alpha_j+\alpha_k
\,=\,\pi,\label{E-2-5c}
\end{align}
\end{subequations}
where $a,b,c\in(0,\pi)$ (resp. $a,b,c\in(0,\infty)$) are constants in spherical (resp. Euclidean) geometry.
\end{lemma}

\begin{proof}
In spherical geometry, the law of cosines gives
\[
\cos\alpha_i\,=\,\frac{\cos l_i-\cos l_j\cos l_k}{\sin l_j\sin l_k}.
\]
As $(r_i,r_j,r_k)\rightarrow(0,a,b)$, we have $l_j\to b,$  $l_k\to a$
and
\[
\cos l_i\,\to\, \cos a\cos b-\cos \Theta_i\sin a\sin b.
\]
Consequently,
\[
\cos\alpha_i\,\to\, -\cos \Theta_i,
\]
which implies~\eqref{E-2-5a}.

Let us consider~\eqref{E-2-5b}. A routine computation indicates
\[
\cos \alpha_i+\cos \alpha_j\,=\, \frac{\sin (l_i+l_j)\big[\cos(l_i- l_{j})-\cos l_{k}\big]}{\sin l_{i}\sin l_j\sin l_{k}}.
\]
Therefore,
\[
\begin{aligned}
 |\cos \alpha_i+\cos \alpha_j|\,\leq\,&\,\frac{|\sin (l_i+l_j)|(1-\cos l_{k})}{\sin l_{i}\sin l_j\sin l_{k}}\\
\,=\,&\,\frac{|\sin (l_i+l_j)|\sin(l_k/2)}{\sin l_{i}\sin l_j\cos( l_{k}/2)}.
\end{aligned}
\]
As $(r_i,r_j,r_k)\to (0,0,c)$, we obtain
\[
l_i\,\to\, c,\;\; l_j\,\to\, c,\;\; l_k\,\to\, 0.
\]
It follows that
\[
\cos \alpha_i+\cos \alpha_j \,\to\, 0,
\]
which yields
\[
\alpha_i+\alpha_j\, \to\, \pi.
\]

It remains to prove~\eqref{E-2-5c}. As $(r_i,r_j,r_k)\to (0,0,0)$, the area of the triangle tends to zero. Applying the Gauss-Bonnet Formula, we get the required result.

In Euclidean geometry, the lemma has been proved in the work of Ge-Hua-Zhou~\cite{Ge-Hua-Zhou}. In fact, formula~\eqref{E-2-5a} is straightforward. To show~\eqref{E-2-5b}, consider the formula
\[
\cos\alpha_k\,=\,\frac{r_{k}^2+r_{j}r_{k}\cos\Theta_{i}+r_{i}r_{k}\cos\Theta_{j}-r_{i}r_{j}\cos\Theta_{k}}{\sqrt{r_{j}^2+r_{k}^2+2r_{j}r_{k}\cos\Theta_{i}}
\sqrt{r_{i}^2+r_{k}^2+2r_{i}r_{k}\cos\Theta_{j}}}.
\]
As $(r_i,r_j,r_k)\to(0,0,c)$, it is easy to see
\[
\cos \alpha_k \,\to\, 1.
\]
Thus $\alpha_k\to 0$, which asserts~\eqref{E-2-5b}. Formula~\eqref{E-2-5c} is trivial.
\end{proof}

\begin{lemma}\label{L-2-6}
Given three mutually intersecting disks $D_i, D_j, D_k$ on the Riemann sphere $\hat{\mathbb C}$ with exterior intersection angles $\Theta_i,\Theta_j,\Theta_k\in[0,\pi)$, suppose $D_i\cap D_j\subset D_k$. Then
\[
\Theta_i+\Theta_j \, \geq\, \Theta_k + \pi.
\]
Particularly, when $D_i,D_j$ intersect at only one point, we have
\[
\Theta_i +\Theta_j\,\geq\, \pi,
\]
where the equality holds if and only if $\partial D_i\cap \partial D_j\cap\partial D_k\neq\emptyset$.
\end{lemma}

\begin{figure}[htbp]\centering
\includegraphics[width=0.81\textwidth]{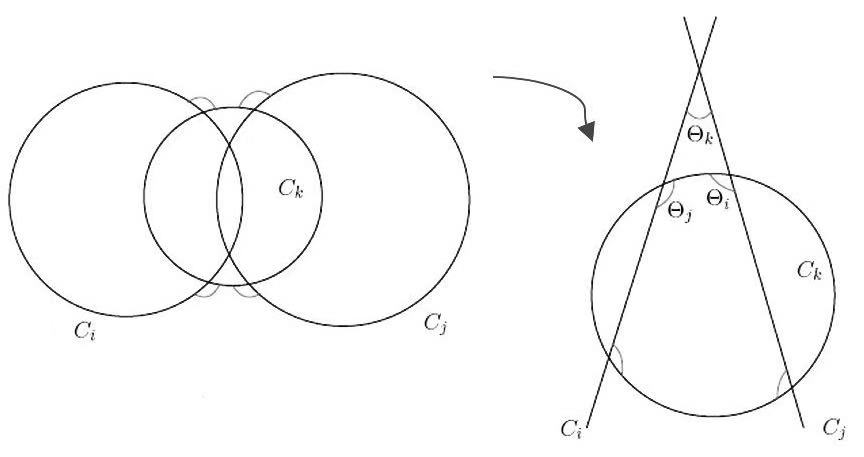}
\caption{The angle relation}\label{F-5}
\end{figure}

\begin{proof}
Using an appropriate M\"{o}bius transformation, we assume the boundaries  $\partial D_i, \partial D_j$ become a pair of intersecting lines. As in Figure~\ref{F-5}, it is easy to see
\[
\pi-\Theta_i+\pi-\Theta_j+\Theta_k\, \leq \, \pi.
\]
Hence
\[
\Theta_i+\Theta_j\, \geq \, \pi+\Theta_k,
\]
where the equality holds if and only if
$\partial D_i\cap \partial D_j\cap \partial D_k\neq\emptyset$. The remainder part of the proof is trivial. We omit the details.
\end{proof}

\begin{lemma}\label{L-2-7}
Given three mutually intersecting disks $D_i, D_j, D_k$ on the Riemann sphere
$\hat{\mathbb C}$ with exterior intersection angles $\Theta_i,\Theta_j,\Theta_k\in[0,\pi)$, suppose
$\mathbb{D}_i\cup \mathbb{D}_j\cup \mathbb{D}_k\subsetneq\hat{\mathbb C}$, where $\mathbb D_\mu$ denotes the interior of
$D_\mu$ for $\mu=i,j,k$. If $\Theta_i+\Theta_j+\Theta_k<\pi$, then
\[
D_i\cap D_j\cap D_k\,=\,\emptyset.
\]
\end{lemma}
\begin{proof}
First we claim the interior of the set $\mathcal I=\hat{\mathbb C}\setminus(\mathbb{D}_i\cup \mathbb{D}_j\cup \mathbb{D}_k)$ is non-empty. Otherwise, at least one connected component of $\mathcal I$ is an arc or a point, which yields $\Theta_i+\Theta_j+\Theta_k=\pi$. This contradicts the condition. Using an appropriate M\"{o}bius transformation, we may assume the infinity is an interior point of $\mathcal I$. In this way we regard $D_i, D_j, D_k$ as disks on the complex plane $\mathbb C$. Let $r_i,r_j,r_k$ be the radii of $D_i, D_j, D_k$, respectively. For each $t\in[0,1]$, by Lemma~\ref{L-2-4}, there are three mutually intersecting disks $D_{i}(t),D_{j}(t),D_{k}(t)$ on $\mathbb C$ realizing the data $(r_i,r_j,r_k,t\Theta_i,t\Theta_j,t\Theta_k)$ and satisfying  $D_{\mu}(1)=D_\mu$ for $\mu=i,j,k$.

Now suppose the lemma is not true. We have
\[
\cap_{\mu=i,j,k}D_\mu(1)\,=\,\cap_{\mu=i,j,k}D_\mu\neq\emptyset.
\]
Meanwhile, it is easy to see
\[
\cap_{\mu=i,j,k}D_\mu(0)\,=\,\emptyset.
\]
That means there exists $t_0\in[0,1]$ such that $\cap_{\mu=i,j,k}D_\mu(t_0)$ consists of a point. Hence
\[
t_0(\Theta_i+\Theta_j+\Theta_k)\,=\,\pi,
\]
which also contradicts the condition.
\end{proof}

\section{Patterns of circles without interstices}\label{S-3}
For a circle pattern $\mathcal P=\{D_v\}_{v\in V}$ on the Riemann sphere $\hat{\mathbb C}$, we call each connected component of the set $\hat{\mathbb C}\setminus (\cup_{v\in V}D_v)$ an interstice. In this section we focus on circle patterns which possess no interstice. According to Thurston's observation, these patterns are closely related to convex hyperbolic polyhedra of finite-volume. Our purpose is the following special case of Theorem~\ref{T-1-4}.

\begin{theorem}\label{T-3-1}
Let $\mathcal T$ be a triangulation of the sphere with more than four vertices. Assume that $\Theta:E\to(0,\pi)$ is a function satisfying conditions $\mathrm{\mathbf{(c1)}},\mathrm{\mathbf{(c2)}},
\mathrm{\mathbf{(c3)}},\mathrm{\mathbf{(c4)}}$ and the condition below:
\begin{itemize}
\item[$\mathrm{\mathbf{(m5)}}$] If $e_1,e_2,e_3$ form the boundary of a triangle of $\mathcal T$, then $\sum_{\mu=1}^3\Theta(e_\mu)\geq \pi$.
\end{itemize}
Then there exists an \textbf{irreducible} circle pattern $\mathcal P$ on the Riemann sphere $\hat{\mathbb C}$ with contact graph isomorphic to the $1$-skeleton of $\mathcal T$ and exterior intersection angles given by $\Theta$.
\end{theorem}

We will utilize spherical geometry to study these circle patterns.
To begin with, we endow $\hat{\mathbb C}$ with the following Riemannian metric
\[
\mathrm{d s}\,=\,\frac{2|\mathrm{d}z|}{1+|z|^2}.
\]
Note that $(\hat{\mathbb C},\mathrm{d s})$ is isometric to the unit sphere $\mathbb S^2$ in $\mathbb R^3$. In what follows we shall not distinguish  $(\hat{\mathbb C},\mathrm{d s})$ and $\mathbb S^2$ for the sake of simplicity.

\subsection{Configuration spaces}
We wish to produce the desired circle pattern via configurations spaces. The framework of this approach has been rooted in the works of Bauer-Stephenson-Wegert~\cite{Bauer-Stephenson-Wegert}, Zhou~\cite{Zhou}, Bowers-Bowers-Pratt~\cite{Bowers-Bowers-Pratt} and Connelly-Gortler~\cite{Connelly-Gortler} in consideration of computational mechanisms (see~\cite{Bauer-Stephenson-Wegert,Connelly-Gortler}) or understanding of local rigidity (see~\cite{Zhou,Bowers-Bowers-Pratt}).

Let $V,E,F$ denote the sets of vertices, edges and triangles of $\mathcal T$.  Set $M=\hat{\mathbb C}^{|V|}\times (0,\pi)^{|V|}$. Then $M$ is a smooth manifold parameterizing the space of patterns of $|V|$ disks on $(\hat{\mathbb C},\mathrm{d}s)$. Since $\mathcal T$ is a triangulation, it is trivial to see
\[
3|F|\,=\,2|E|.
\]
Combining with Euler's Formula, we have
\[
\begin{aligned}
\dim (M)\,&=\,3|V|\,\\
\,&=\,3|V|+\big(3|F|-2|E|\big)\\
\,&=\,3\big(|V|-|E|+|F|\big)+|E|\\
\,&=\,|E|+6.
\end{aligned}
\]

A point $(\mathbf{z},\mathbf{r})=(z_1,\cdots,z_{|V|},r_1,\cdots,r_{|V|})\in M$ is called a configuration, since it assigns each vertex $v_i\in V$ a closed disk $D_i$, where $D_i$ is centered at $z_i$ and is of radius $r_i$. For each edge $e=[v_i,v_j]\in E$, the inversive distance is defined to be
\[
I(e,\mathbf{z},\mathbf{r})\,=\,\frac{\cos r_i\cos r_j-\cos d(z_i,z_j)}{\sin r_i\sin r_j},
\]
where $d(z_i,z_j)$ denotes the distance between $z_i$ and $z_j$. The subspace $M_{E}\subset M$ is the set of configurations subject to
\[
-1\,<\,I(e,\mathbf{z},\mathbf{r})\,<\,1
\]
for every $e\in E$. Let $M_{\mathcal T}\subset M_{E}$ be the subspace of configurations which give $\mathcal T$-type circle patterns. More precisely, $(\mathbf{z},\mathbf{r})\in M_{\mathcal T}$ if and only if there is a geodesic triangulation $\mathcal T(\mathbf{z},\mathbf{r})$ of  $(\hat{\mathbb C},\mathrm{d s})$ with the properties below:
\begin{itemize}
\item[$\mathrm{\langle \mathbf{x1}\rangle}$] $\mathcal T(\mathbf{z},\mathbf{r})$ is isotopic to $\mathcal T$;
\item[$\mathrm{\langle \mathbf{x2}\rangle}$]The vertices of $\mathcal T(\mathbf{z},\mathbf{r})$ coincide with the centers $z_1,z_2,\cdots,z_{|V|}$.
\end{itemize}
We define $M_{G}\subset M_{\mathcal T}$ to be the subspace of configurations under the further restriction:
\begin{itemize}
\item[$\mathrm{\langle \mathbf{x3}\rangle}$] The disks
$D_\alpha, D_\beta$ are disjoint whenever there is no edge between $v_\alpha$ and $v_\beta$.
\end{itemize}
Obviously, each $(\mathbf{z},\mathbf{r})\in M_{G}$ gives a circle pattern realizing the $1$-skeleton of $\mathcal T$ as contact graph. Now let $M_{IG}\subset M_{G}$ consist of configurations satisfying the \textbf{irreducible} property:
\begin{itemize}
\item[$\mathrm{\langle \mathbf{x4}\rangle}$] When $A$ is a proper subset of $V$, then $\cup_{v_i\in A} D_i\subsetneq \hat{\mathbb C}$.
\end{itemize}
Note that $M_E, M_{\mathcal T},M_{G}, M_{IG}$ are all non-empty open subsets of $M$ and thus are smooth manifolds of dimension $|E|+6$. Here the non-emptiness follows from Theorem~\ref{T-1-1}.

Choose a triangle $\triangle^\star$ of $\mathcal T$. Let $v_a,v_b,v_c$ be the vertices of $\triangle^\star$ and let $e_a,e_b,e_c$ be the edges opposite to $v_a,v_b,v_c$, respectively. We use $M_{IG}^\star\subset M_{IG}$ to represent the subspace of configurations such that the following normalization conditions hold:
\begin{itemize}
\item[$\mathrm{\langle \mathbf{x5}\rangle}$]
    $z_a\,=\,0$,\;\, $z_b\,>\,0$,\;\, $0\,<\,\operatorname{Arg}(z_c)\,<\,\pi$;
\item[$\mathrm{\langle \mathbf{x6}\rangle}$]
$r_a\,=\,r_b\,=\,r_c\,=\,\pi/4$.
\end{itemize}
It is easy to see $M^\star_{IG}$ is a smooth manifold of dimension $|E|$. In addition, we have the following smooth map
\[
\begin{aligned}
\mathcal{E}v:\quad &M^\star_{IG} &\longrightarrow \quad &\quad\;  Y:=(0,\pi)^{|E|}\\
&(\mathbf{z},\mathbf{r})  &\longmapsto \quad  & \big(\Theta(e_1,\cdot),\Theta(e_2,\cdot),\cdots\big),
\end{aligned}
\]
where $\Theta(e,\cdot)=\arccos I(e,\cdot)$.

Below is an important property regarding to $\mathcal T$-type circle patterns. For simplicity, we shall write the interior of every disk $D_i$ as $\mathbb D_i$.

\begin{lemma}\label{L-3-2}
Suppose $\mathcal P$ is a $\mathcal T$-type circle pattern on $(\hat{\mathbb C},\mathrm{d s})$ whose exterior intersection angle function $\Theta: E\to(0,\pi)$ satisfies condition $\mathrm{\mathbf{(c1)}}$.  Let $v_\alpha,v_\beta\in V$ be a pair of non-adjacent vertices. For any $p\in D_\alpha\cap D_\beta$,  there exists  $v_\eta\in V\setminus\{v_\alpha,v_\beta\}$ such that $p\in\mathbb D_\eta$.
\end{lemma}

\begin{figure}[htbp]\centering
\includegraphics[width=0.57\textwidth]{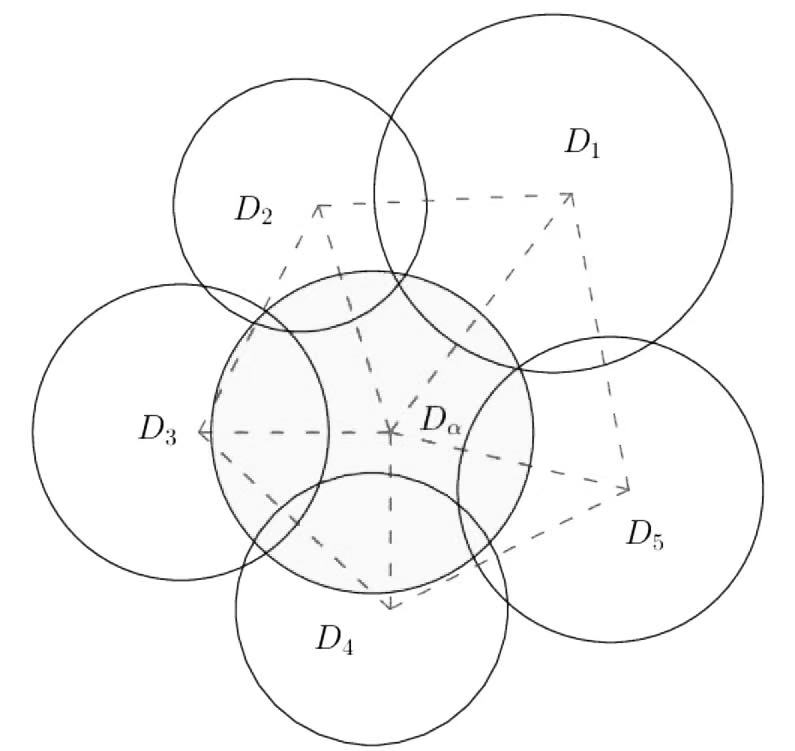}
\caption{A flower of disks}
\end{figure}

\begin{proof}
Let $S\big(\{v_\alpha\}\big)$ denote the union of \textbf{open} $d$-simplices $(d=0,1,2)$ of $\mathcal T$ incident to $v_\alpha$. Suppose
$v_{1},\cdots, v_{m}\in V$, in anticlockwise order, are all adjacent vertices of $v_\alpha$. We claim
\begin{equation}\label{E-3-1}
 D_\alpha\,\subset\,\big(\cup_{\mu=1}^m \mathbb D_{\mu}\big)\cup S\big(\{v_\alpha\}\big)\,:=\,\operatorname{\mathbb{FL}}_\alpha.
\end{equation}

We first observe  $\partial D_\alpha\subset \operatorname{\mathbb{FL}}_\alpha$.
Otherwise, there exists $\mu_0\in\{1,\cdots,m\}$ such that
\[
D_{\mu_0}\cap D_{\mu_0+1}\,\subset\, D_\alpha.
\]
By Lemma~\ref{L-2-6}, we get
\[
\Theta\big([v_{\mu_0},v_\alpha]\big)+\Theta\big([v_\alpha,v_{\mu_0+1}]\big)\, \geq\,  \pi+\Theta\big([v_{\mu_0},v_{\mu_0+1}]\big),
\]
which contradicts condition $\mathrm{\mathbf{(c1)}}$.

Since $\Theta>0$, it is easy to see $\partial S\big(\{v_\alpha\}\big)\subset \operatorname{\mathbb{FL}}_\alpha $. Writing the equations of the closed curves $\partial D_\alpha, \partial S\big(\{v_\alpha\}\big)$ with polar coordinates and considering their convex combination, we construct a family of closed Jordan regions $\{J_t\}_{0\leq t\leq 1}$ with the following properties:
 \[
J_0\,=\,\overline{S\big(\{v_\alpha\}\big)},\quad\, J_1\,=\,D_\alpha,\quad\, \partial J_t\,\subset\,\operatorname{\mathbb{FL}}_\alpha.
\]
Let $X$ be the set of $t\in[0,1]$ satisfying
\[
J_t\,\subset\,\operatorname{\mathbb{FL}}_\alpha.
\]
A routine analysis deduces that $X$ is a non-empty, open and closed subset of $[0,1]$. As a result, $X=[0,1]$. Taking $t=1$, we prove the claim.

We now assume, by contradiction, that there exists a point
$p_0\in D_\alpha\cap D_\beta$ which does not belong to any open disks of $\mathcal P$ except for $\mathbb D_\alpha, \mathbb D_\beta$. Consequently,
\[
p_0\notin \cup_{\mu=1}^m \ \mathbb D_{\mu}.
\]
Meanwhile, relation~\eqref{E-3-1} gives
 \[
 p_0 \,\in\, D_\alpha\ \subset\big(\cup_{\mu=1}^m \mathbb D_{\mu}\big)\cup S\big(\{v_\alpha\}\big).
\]
Hence
\[
p_0\, \in\,S\big(\{v_\alpha\}\big).
\]
Similarly, we obtain
\[
p_0\, \in\, S\big(\{v_\beta\}\big).
\]
That means $S\big(\{v_\alpha\}\big)\cap S\big(\{v_\beta\}\big)$ is non-empty, which occurs if and only if there exists an edge between $v_\alpha$ and $v_\beta$. This contradicts that $v_\alpha,v_\beta$ is a pair of non-adjacent vertices.
\end{proof}


\subsection{Topological degree}
In order to prove Theorem~\ref{T-3-1}, we need to demonstrate that any function $\Theta: E\to(0,\pi)$ satisfying conditions $\mathrm{\mathbf{(c1)}}-\mathrm{\mathbf{(c4)}}$ and $\mathrm{\mathbf{(m5)}}$ is in the image of the map $\mathcal{E}v$.
 Recall that
\[
\dim(M^\star_{IG})\,=\,\dim (Y)\,=\,|E|.
\]
To this end, we shall make the use of the topological degree theory. Compared with the continuity method used by Thurston~\cite[Chap.13]{Thurston}, this approach has the advantage of dealing with existence and rigidity in a separated manner.

The objective is to find a relatively compact open subset (i.e. an open subset whose closure is compact) $\Lambda\subset M^\star_{IG}$ and determine the degree $\deg(\mathcal{E}v,\Lambda,\Theta)$. Once we show
\[
\deg(\mathcal{E}v,\Lambda,\Theta)\,\neq\,0,
\]
a basic property (see Theorem~\ref{T-6-9}) of topological degree yields  there exists $(\mathbf{z},\mathbf{r})\in \Lambda$ such that $\mathcal{E}v(\mathbf{z},\mathbf{r})=\Theta$, which concludes Theorem~\ref{T-3-1}.

Let us compute the degree by deforming $\Theta$ to another value which is relatively easier to manipulate. Suppose $W_{m}$ is the set of functions satisfying conditions $\mathrm{\mathbf{(c1)}}-\mathrm{\mathbf{(c4)}}$ and $\mathrm{\mathbf{(m5)}}$. Note that one needs to prove nothing if $W_m$ is an empty set. So from now on we only consider the case that $W_{m}$ is non-empty. The following elementary fact indicates there is an ample scope for deformation.

\begin{lemma}\label{L-3-3}
If $W_m\neq \emptyset$, then $W_m\cap (0,\pi/2)^{|E|}\neq \emptyset$.  Furthermore, $W_m\cap (0,\pi/2)^{|E|}$ has positive measure in $(0,\pi)^{|E|}$.
\end{lemma}
\begin{proof}
Choose $\Theta\in W_m$. For $s\in(0,1]$, set
\[
\Theta_s\,=\,(1-s)\psi+s\Theta,
\]
where
\[
\psi\,=\,(\pi/3,\pi/3,\cdots,\pi/3).
\]
Taking $s_0$ sufficiently close to zero, we have $\Theta_{s_0}\in W_m\cap (0,\pi/2)^{|E|}$, which implies
\[
W_m\cap (0,\pi/2)^{|E|}\,\neq\,\emptyset.
\]
For $\epsilon>0$ , let $U_\epsilon$ consist of vectors $\theta\in \mathbb R^{|E|}$ satisfying
\[
\Theta_{s_0}(e)\,<\,\theta(e)\,<\,\Theta_{s_0}(e)+\epsilon,\quad \forall\,e\in E.
\]
When $\epsilon$ is sufficiently small, it is easy to see $U_\epsilon\subset W_m\cap (0,\pi/2)^{|E|}$. Thus $W_m\cap(0,\pi/2)^{|E|}$ has positive measure in $(0,\pi)^{|E|}$.
\end{proof}

Applying Sard's Theorem, we can find a regular value $\widetilde{\Theta}\in U_\epsilon\subset W_m\cap(0,\pi/2)^{|E|}$ of the map $\mathcal{E}v$. For $t\in[0,1]$,  set
\[
\Theta(t)\,=\,(1-t)\widetilde{\Theta}+t\Theta.
\]
Evidently, each $\Theta(t)$ satisfies conditions $\mathrm{\mathbf{(c1)}}-\mathrm{\mathbf{(c4)}}$ and $\mathrm{\mathbf{(m5)}}$. That means $\big\{\Theta(t)\big\}_{0\leq t\leq 1}$ form a continuous curve $\gamma$ in $W_m$. We will need the following technical result.

\begin{lemma}\label{L-3-4}
There exists a relatively compact open subset $\Lambda\subset M^\star_{IG}$ such that
\[
\mathcal{E}v^{-1}(\gamma)\,\subset\,\Lambda.
\]
\end{lemma}

It suffices to show that any sequence of configurations $\big\{(\mathbf{z}_n,\mathbf{r}_n)\big\}\subset \mathcal{E}v^{-1}(\gamma)$ contains a convergent subsequence in $M^\star_{IG}$. Note that each  $(\mathbf{z}_n,\mathbf{r}_n)$ gives a normalized \textbf{irreducible} circle pattern $\mathcal P_n=\{D_{i,n}\}_{i=1}^{|V|}$ on the Riemann sphere $\hat{\mathbb C}$ realizing the data $(\mathcal T,\Theta_n)$, where
\[
\Theta_n\,=\,(1-t_n)\widetilde{\Theta}+t_n\Theta
\]
for some $t_n\in[0,1]$. Following Thurston~\cite{Thurston}, we define the apex curvature $K_{i,n}$ to be
\[
K_{i,n}\,=\,2\pi -\sigma_{i,n}.
\]
Here $\sigma_{i,n}$ denotes the cone angle at $v_i$, which is equal to the sum of inner angles at $v_i$ for all triangles of
$\mathcal T(\mathbf{z}_n,\mathbf{r}_n)$ incident to $v_i$. For $i=1,2,\cdots,|V|$, we have
\begin{equation}\label{E-3-2}
K_{i,n}\,=\,0.
\end{equation}

Before the proof of Lemma~\ref{L-3-4}, let us show the following proposition.
\begin{proposition}\label{P-3-5}
There is a subsequence of $\big\{(\mathbf{z}_{n},\mathbf{r}_{n})\big\}$  converging to a point $(\mathbf{z}_\infty,\mathbf{r}_\infty)\in M$.
\end{proposition}

\begin{proof}
First we can extract a subsequence $\big\{(\mathbf{z}_{n_k},\mathbf{r}_{n_k})\big\}$ of $\big\{(\mathbf{z}_{n},\mathbf{r}_{n})\}$ convergent to some point $(\mathbf{z}_\infty,\mathbf{r}_\infty)\in\hat{\mathbb C}^{|V|}\times[0,\pi]^{|V|}$. Let $\mathcal P_\infty=\{D_{i,\infty}\}_{i=1}^{|V|}$ denote the pattern given by $(\mathbf{z}_\infty,\mathbf{r}_\infty)$.
Assume the proposition is not true. Then there exists $i_0\in \{1,2,\cdots,|V|\}$ satisfying
\[
r_{i_0,n_k}\,\to\,\pi\quad\;\mathrm{\text{or}}\quad\;r_{i_0,n_k}\,\to\,0.
\]

In the former case,
$\hat{\mathbb C}\setminus \mathbb D_{i_0,n_k}$ tends to a point
$x\in \hat{\mathbb C}$ in the sense of Hausdorff convergence. If $x$ is an interior point of
$D_{a,\infty}\cup D_{b,\infty}\cup D_{c,\infty}$, for $n_k$ sufficiently large we have
\[
 D_{a,n_k}\cup D_{b,n_k}\cup D_{c,n_k}\cup D_{i_0,n_k}\,=\,\hat{\mathbb C}.
\]
Since $\mathcal T$ possesses more than four vertices, this contradicts that  $\mathcal P_{n_k}$ is \textbf{irreducible}. If $x$ is not an interior point of $D_{a,\infty}\cup D_{b,\infty}\cup D_{c,\infty}$, then one of the disks $D_{a,n_k},D_{b,n_k},D_{c,n_k}$ is covered by $\mathbb D_{i_0,n_k}$ when $n_k$ is sufficiently large. Without loss of generality, assume $D_{a,n_k}\subset \mathbb D_{i_0,n_k}$. That means $D_{a,n_k}\cap D_{i_0,n_k}=D_{a,n_k}$.
To obtain a contradiction, we claim $D_{a,n_k}\cap D_{i_0,n_k}$ is actually a proper subset of $D_{a,n_k}$. Indeed, if $v_{i_0}$ is not an adjacent vertex of $v_a$,  the claim follows straightforwardly from the assumption that $D_{a,n_k},D_{i_0,n_k}$ are disjoint; If $v_{i_0}$ is an adjacent vertex of $v_a$, the condition $\Theta_{n_k}([v_{i_0},v_a])\in(0,\pi)$ also implies $D_{a,n_k}\cap D_{i_0,n_k}\subsetneq D_{a,n_k}$.

In the latter case, by taking a subsequence, if necessary, we assume  $t_{n_k}$ converges to some $t_\infty\in[0,1]$.  Let $V_0\subset V$ be the set of vertices $v_i\in V$ for which $r_{i,n_k}\to 0$. Then $V_0$ is a non-empty subset of $V\setminus\{v_a,v_b,v_c\}$. For $A\subset V_0$, let $S(A)$ denote the union of \textbf{open} $d$-simplices $(d=0,1,2)$ of $\mathcal T$ incident to at least one vertex in $A$ and let $Lk(A)$ denote the set of pairs $(e,u)$ of an edge $e\in E$ and a vertex $u\in V$ with the following properties:
\[
(i)\, \partial e\cap V_0=\emptyset;\quad(ii)\,u\in A ;\quad (iii)\, e\; \mathrm{\text{and}}\;u\; \mathrm{\text{form  a triangle of} }\; \mathcal T.
\]
The following Proposition~\ref{P-3-6} indicates
\[
\sum\nolimits_{v_i\in V_0}K_{i,n_k}\,\to\, -\sum\nolimits_{(e,u)\in Lk(V_0)}\big[\pi-(1-t_\infty)\widetilde{\Theta}(e)-t_\infty\Theta(e)\big]+2\pi\chi\big(S(V_0)\big).
\]
Together with~\eqref{E-3-2}, we obtain
\begin{equation}\label{E-3-3}
0\,=\,-\sum\nolimits_{(e,u)\in Lk(V_0)}\big[\pi-(1-t_\infty)\widetilde{\Theta}(e)-t_\infty\Theta(e)\big]
+2\pi\chi\big(S(V_0)\big).
\end{equation}
Let $\{S_\tau\}_{\tau=1}^N$ be the set of connected components of $S(V_0)$ and let $A_\tau\subset V_0$ be the subset of vertices covered by $S_\tau$. Then
\[
\chi\big(S(V_0)\big)\,=\,\sum\nolimits_{\tau=1}^N\chi(S_\tau),
\quad Lk(V_0)\,=\,\cup_{\tau=1}^N Lk(A_\tau).
\]
Taking into consideration~\eqref{E-3-3}, we assert there exists  $\tau_0\in\{1,\cdots,N\}$ satisfying
\begin{equation}\label{E-3-4}
-\sum\nolimits_{(e,u)\in Lk(A_{\tau_0})}\big[\pi-(1-t_\infty)\widetilde{\Theta}(e)-t_\infty\Theta(e)\big]
+2\pi\chi(S_{\tau_0})
\,\geq\,0.
\end{equation}

In case that $S_{\tau_0}$ is $h$-connected such that $h\geq 2$, we get $\chi(S_{\tau_0})=2-h\leq 0$, which yields
\[
-\sum\nolimits_{(e,u)\in Lk(A_{\tau_0})}\big[\pi-(1-t_\infty)\widetilde{\Theta}(e)-t_\infty\Theta(e)\big]
+2\pi\chi(S_{\tau_0})\,<\,0.
\]
This contradicts~\eqref{E-3-4}.

In case that $S_{\tau_0}$ is simply connected, then $\chi(S_{\tau_0})=1$. Writing $Lk(A_{\tau_0})=\big\{(e_\mu,u_\mu)\big\}_{\mu=1}^k$, we reduce inequality~\eqref{E-3-4} to
\begin{equation}\label{E-3-5}
\sum\nolimits_{\mu=1}^k (1-t_\infty)\widetilde{\Theta}(e_\mu)+\sum\nolimits_{\mu=1}^k t_\infty\Theta(e_\mu)\,\geq \,(k-2)\pi.
\end{equation}
Note that $e_1,\cdots,e_k$ form a (possibly not simple) closed curve $\Gamma$ bounding the surface $S_{\tau_0}$. Apparently, $k\geq 3$.  Let $\mathbb J$ denote the connected component of the complement $\hat{\mathbb C}\setminus\Gamma$ such that $\operatorname{Int}(\triangle^\star)\subset \mathbb J$, where $\operatorname{Int}(\triangle^\star)$ is the interior of the marked triangle $\triangle^\star$.

Set $J=\mathbb J\cup \partial \mathbb J$. If $J$ includes more than two triangles of $\mathcal T$, Lemma~\ref{L-2-1} indicates
\[
\sum\nolimits_{\mu=1}^{k_0} \Theta(e_\mu)\,<\,(k_0-2)\pi,
\]
where $e_1,\cdots,e_{k_0}$ form the simple closed curve bounding $\mathbb J$. Hence
\[
\sum\nolimits_{\mu=1}^{k} \Theta(e_\mu)\,=\,\sum\nolimits_{\mu=1}^{k_0} \Theta(e_\mu)+\sum\nolimits_{\mu=k_0+1}^{k} \Theta(e_\mu)\,<\,(k_0-2)\pi+(k-k_0)\pi\,=\,(k-2)\pi.
\]
Recall that $\widetilde{\Theta}\in W_m\cap(0,\pi/2)^{|E|}$. We derive
\[
\sum\nolimits_{\mu=1}^k (1-t_\infty)\widetilde{\Theta}(e_\mu)+\sum\nolimits_{\mu=1}^k t_\infty\Theta(e_\mu)\,<\,(1-t_\infty)(k-2)\pi+t_\infty(k-2)\pi\,=\,(k-2)\pi,
\]
which contradicts~\eqref{E-3-5}.

If $J$ includes at most two triangles of $\mathcal T$, then $J=\triangle^\star$ or $J$ is the union of two adjacent triangles such that one of them equal to $\triangle^\star$. No matter which case occurs, among $e_a,e_b,e_c$ there are at least two edges forming a part of $\partial \mathbb J$. Without loss of generality, assume $e_a\subset \partial \mathbb J$ and $e_b\subset \partial \mathbb J$. We then find vertices $v_{i_a},v_{i_b}$ such that $(e_a,v_{i_a}),(e_b,v_{i_b})\in Lk(A_{\tau_0})$. As in Figure~\ref{F-7}, define $q_{a,n_k}=\partial D_{b,n_k}\cap\partial D_{c,n_k}\cap D_{i_a}$  and $q_{b,n_k}=\partial D_{a,n_k}\cap\partial D_{c,n_k}\cap D_{i_b}$. Clearly,
\[
d(q_{a,n_k},z_{i_a,n_k})\,\leq\,r_{i_a,n_k},\quad\;d(q_{b,n_k},z_{i_b,n_k})\,\leq\,r_{i_b,n_k}.
\]
Noting that $v_{i_a},v_{i_b}\in A_{\tau_0}\subset V_0$, we have $r_{i_a,n_k}\to 0$ and $r_{i_b,n_k}\to 0$, which indicates
\[
d(q_{a,n_k},z_{i_a,n_k})\,\to\,0,\quad\; d(q_{b,n_k},z_{i_b,n_k})\,\to\,0.
\]
Because $S_{\tau_0}$ is connected, there is a sequence of vertices $v_{\zeta_1}=v_{i_a},v_{\zeta_2},\cdots,v_{\zeta_m}=v_{i_b}\in A_{\tau_0}$ such that $v_{\zeta_\varrho}$ and $v_{\zeta_{\varrho+1}}$ are adjacent. It is easy to see
\[
d(z_{i_a,n_k},z_{i_b,n_k})\,\to\,0.
\]
The above formulas together yield
\[
d(q_{a,n_k},q_{b,n_k})\,\to\,0.
\]
Namely,
\[
q_{a,\infty}\,=\,q_{b,\infty}.
\]
However, under normalization condition
$\mathrm{\mathbf{\langle x6\rangle}}$, a routine computation implies
\[
d(q_{a,\infty},q_{b,\infty})\,>\,0.
\]
It is a contradiction. We thus finish the proof.
\end{proof}

\begin{figure}[htbp]\centering
\includegraphics[width=0.64\textwidth]{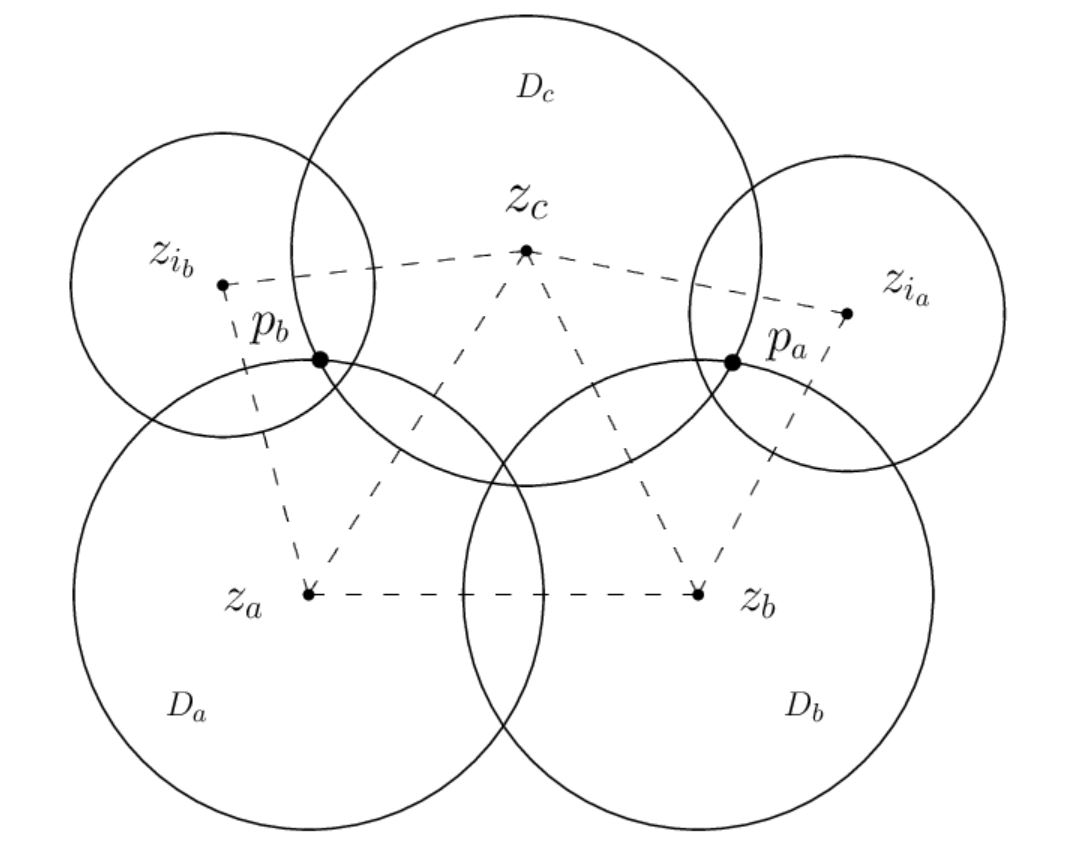}
\caption{Some marked points}\label{F-7}
\end{figure}

The following proposition, required in the above proof, synthesizes the information on apex curvatures as some disks degenerate to points. It is worth mentioning that the Euclidean version of this result has appeared in the work of Thurston~\cite[Chap. 13]{Thurston}.

\begin{proposition}\label{P-3-6}
Let $S(V_0)$ and $Lk(V_0)$ be as above. Then
\[
\sum_{v_i\in V_0}K_{i,n_k}\,\to\,
 -\sum_{(e,u)\in Lk(V_0)}\big[\pi-(1-t_\infty)\widetilde{\Theta}(e)-t_\infty\Theta(e)\big]+2\pi\chi\big(S(V_0)\big).
\]
\end{proposition}

\begin{proof}
For $m=1,2,3$, let $F_m(V_0)$ denote the set of triangles having exactly $m$ vertices in $V_0$. By Lemma~\ref{L-2-5}, it follows from~\eqref{E-2-5a},~\eqref{E-2-5b} and~\eqref{E-2-5c} that
\[
\sum_{v_i\in V_0}K_{i,n_k}\,\to\,2\pi|V_0| -\sum_{(e,u)\in Lk(V_0)}\big[\pi-(1-t_\infty)\widetilde{\Theta}(e)-t_\infty\Theta(e)\big]-\pi|F_2(V_0)|-\pi|F_3(V_0)|.
\]
Let $E(V_0), F(V_0)$ be the sets of edges and triangles having at least one vertex in $V_0$. We are ready to see
\[
|F(V_0)|\,=\,|F_1(V_0)|+|F_2(V_0)|+|F_3(V_0)|.
\]
Meanwhile, notice that
\[
|F_1(V_0)|\,=\,|Lk(V_0)|
\]
and
\[
3|F(V_0)|\,=\,2|E(V_0)|+|Lk(V_0)|.
\]
Combining the above relations gives
\[
\begin{aligned}
   &2|V_0|-|F_2(V_0)|-|F_3(V_0)|\\
=\,&2|V_0|-|F(V_0)|+|F_1(V_0)|-\big(|F_1(V_0)|-|Lk(V_0)|\big)\\
=\,&2|V_0|-|F(V_0)|+|Lk(V_0)|+\big(3|F(V_0)|-2|E(V_0)|-|Lk(V_0)|\big)\\
=\,&2\big(|V_0|-|E(V_0)|+|F(V_0)|\big)\\
=\,&2\chi\big(S(V_0)\big).
\end{aligned}
\]
As a result,
\[
\sum_{v_i\in V_0}K_{i,n_k}\,\to\, -\sum_{(e,u)\in Lk(V_0)}\big[\pi-(1-t_\infty)\widetilde{\Theta}(e)-t_\infty\Theta(e)\big]+2\pi\chi\big(S(V_0)\big).
\]
\end{proof}

To demonstrate Lemma~\ref{L-3-4}, we need detailed analysis on the limit configuration. As mentioned before, the idea is analogous to the regularization process in PDE theory.
\begin{proof}[\textbf{Proof of Lemma~\ref{L-3-4}}]
By Proposition~\ref{P-3-5}, there exists a subsequence $\{(\mathbf{z}_{n_k},\mathbf{r}_{n_k})\}$  satisfying
\[
(\mathbf{z}_{n_k},\mathbf{r}_{n_k})\,\to\,(\mathbf{z}_\infty,\mathbf{r}_\infty)\,\in\, M.
\]
Let $\mathcal P_\infty=\{D_{i,\infty}\}_{i=1}^{|V|}$ be the circle pattern given by $(\mathbf{z}_\infty,\mathbf{r}_\infty)$ and let $\Theta_\infty$ be the exterior intersection angle function of $\mathcal P_\infty$. The proof is split into the following steps.

\textbf{Step1}, $(\mathbf{z}_\infty,\mathbf{r}_\infty)\in M_{E}$. Since $\big\{(\mathbf{z}_{n_k},\mathbf{r}_{n_k})\big\}\subset \mathcal Ev^{-1}(\gamma)$,  we have
\[
\Theta(e,\mathbf{z}_{n_k},\mathbf{r}_{n_k})\,=\,(1-t_{n_k})\widetilde{\Theta}(e)+t_{n_k}\Theta(e).
\]
As $n_k\to\infty$, it is easy to see
\[
\Theta_\infty(e)\,=\,\Theta(e,\mathbf{z}_\infty,\mathbf{r}_\infty)
\,=\,(1-t_\infty)\widetilde{\Theta}(e)+t_\infty\Theta(e)\,\in\,(0,\pi),
\]
which asserts $(\mathbf{z}_\infty,\mathbf{r}_\infty)\in M_{E}$.
In addition, we immediately check that the function $\Theta_\infty$ satisfies conditions $\mathrm{\mathbf{(c1)}}-\mathrm{\mathbf{(c4)}}$ and $\mathrm{\mathbf{(m5)}}$.

\textbf{Step2}, $(\mathbf{z}_\infty,\mathbf{r}_\infty)\in M_{\mathcal T}$. One needs to deduce that no triangle of $\mathcal T(\mathbf{z}_{n_k},\mathbf{r}_{n_k})$  degenerates to a point or a line segment. More precisely, the triangle inequalities remain valid as $n_k\to\infty$. By Lemma~\ref{L-2-2} and Lemma~\ref{L-2-3}, we reduce the proof to showing every radius $r_{i,n_k}$ is bounded from below and above by positive constants in $(0,\pi)$, which actually follows from Proposition~\ref{P-3-5}.

\textbf{Step3}, $(\mathbf{z}_\infty,\mathbf{r}_\infty)\in M_{G}$. For each pair of non-adjacent vertices $v_\alpha,v_\beta$, it suffices to verify the disks $D_{\alpha,\infty},D_{\beta,\infty}$ are disjoint. Assume on the contrary that $D_{\alpha,\infty} \cap D_{\beta,\infty}\neq\emptyset$. We divide the situation into the following cases:
\begin{itemize}
\item[$(i)$] $D_{\alpha,\infty}\cap D_{\beta,\infty}$ contains interior points. Then for $n_k$ sufficiently large $ D_{\alpha,n_k}\cap D_{\beta,n_k}$ also contains interior points, which contradicts that $D_{\alpha,n_k}, D_{\beta,n_k}$ are disjoint.
\item[$(ii)$] $D_{\alpha,\infty}\cap D_{\beta,\infty}$ consists of a single point $p$. Recall that $\Theta_\infty$ satisfies condition $\mathrm{\mathbf{(c1)}}$. In view of Lemma~\ref{L-3-2}, there exists $v_\eta\in V\setminus\{v_\alpha,v_\beta\}$ such that
\[
p\,\in\,\mathbb D_{\eta,\infty},
\]
which implies $D_{\alpha,\infty}\cap D_{\eta,\infty}$ and $D_{\eta,\infty}\cap D_{\beta,\infty}$ contain interior points.
If one of the vertices $v_\alpha, v_\beta$ is not adjacent to $v_\eta$, similar arguments to case $(i)$ part lead to a contradiction. Suppose both the vertices $v_\alpha, v_\beta$ are adjacent to $v_\eta$. Notice that
    \[
D_{\alpha,\infty}\cap D_{\beta,\infty}\,=\,\{p\}\,\subset\,\mathbb D_{\eta,\infty}.
    \]
It follows from Lemma~\ref{L-2-6} that
\[
\Theta_\infty\big([v_\alpha,v_\eta]\big)+\Theta_\infty\big([v_\eta,v_\beta]\big)\,>\,\pi.
\]
Nevertheless, the function $\Theta_\infty$ satisfies condition $\mathrm{\mathbf{(c2)}}$, we have
\[
\Theta_\infty\big([v_\alpha,v_\eta]\big)+\Theta_\infty\big([v_\eta,v_\beta]\big)\,\leq\,\pi.
\]
This also leads to a contradiction.
\item[$(iii)$] $D_{\alpha,\infty}\cap D_{\beta,\infty}=\partial D_{\alpha,\infty}=\partial D_{\beta,\infty}$. Provided every $v_\tau\in V\setminus\{v_\alpha,v_\beta\}$ is an adjacent vertex of both $v_\alpha$ and $v_\beta$, then $\mathcal T$ is homeomorphic to the boundary of an $m$-gonal bipyramid. Under the assumption that $D_{\alpha,\infty}\cap D_{\beta,\infty}=\partial D_{\alpha,\infty}=\partial D_{\beta,\infty}$, it is easy to see
\[
\Theta_\infty\big([v_\alpha,v_\tau]\big)+\Theta_\infty\big([v_\tau,v_\beta]\big)\,=\,\pi.
\]
Because $\Theta_\infty=(1-t_\infty)\widetilde{\Theta}+t_\infty\Theta$ and $\widetilde{\Theta}\in(0,\pi/2)^{|E|}$, we obtain
\[
\Theta\big([v_\alpha,v_\tau]\big)+\Theta\big([v_\tau,v_\beta]\big)\,\geq\,\pi.
\]
If $m=3$, this violates  condition $\mathrm{\mathbf{(c2)}}$. If $m\geq 4$, we can find edges $e_1,e_2,e_3,e_4$  forming a simple closed curve which separates the vertices of $\mathcal T$ such that
\[
\sum\nolimits_{\mu=1}^4 \Theta(e_\mu)\,\geq\, 2\pi,
\]
which violates condition $\mathrm{\mathbf{(c4)}}$.

We now assume $v_{\tau_0}\in V\setminus\{v_\alpha,v_\beta\}$ is not an adjacent vertex of $v_\alpha$ without loss of generality. Similar reasoning to  case $(i)$ part yields $\mathbb D_{\tau_0,\infty}\cap D_{\alpha,\infty}=\emptyset$. Meanwhile, considering that $D_{\alpha,\infty}\cap D_{\beta,\infty}=\partial D_{\alpha,\infty}=\partial D_{\beta,\infty}$, we  get $\hat{\mathbb C}\setminus D_{\alpha,\infty}=\mathbb D_{\beta,\infty}$. Hence
\[
\mathbb D_{\tau_0,\infty}\,\subset\,\hat{\mathbb C}\setminus D_{\alpha,\infty}\,=\, \mathbb D_{\beta,\infty},
\]
which implies
\[
\mathbb D_{\tau_0,\infty}\cap\mathbb D_{\beta,\infty}\,=\,\mathbb D_{\tau_0,\infty}.
\]
However, following the proof of Proposition~\ref{P-3-5}, we check that $\mathbb D_{\tau_0,\infty}\cap\mathbb D_{\beta,\infty}$ is a proper subset of $\mathbb D_{\tau_0,\infty}$. It is a contradiction.
\end{itemize}

\textbf{Step4}, $(\mathbf{z}_\infty,\mathbf{r}_\infty)\in M_{IG}$. For each $A\subsetneq V$, we need to show
$\cup_{v_i\in A} D_{i,\infty}\subsetneq\hat{\mathbb C}$. To this end, choose a triangle of $\mathcal T$ with vertices $v_l,v_j\in V$ and $v_h\in V\setminus A$. For every $v_i\in V\setminus\{v_l,v_j,v_h\}$, we claim
\begin{equation}\label{E-3-6}
D_{l,\infty}\cap D_{j,\infty}\cap D_{h,\infty}\cap D_{i,\infty}\,= \,\emptyset.
\end{equation}
 In fact, if one of the vertices $v_l, v_j,v_h$ is not adjacent to $v_i$, the claim is a consequence of the third step. Suppose $v_l, v_j,v_h$ are all adjacent to $v_i$. Since $\mathcal T$ possesses more than four vertices, there is no loss of generality in assuming that the edges $[v_l,v_j],[v_j ,v_i], [v_i,v_l]$ form a simple closed curve separating the vertices of $\mathcal T$. The following Proposition~\ref{P-3-7} also implies the claim.

In light of~\eqref{E-3-6}, the point $q=\partial D_{l,\infty}\cap\partial D_{j,\infty}\cap D_{h,\infty}$ does not belong to any disks in $\mathcal P_\infty$ except for
$D_{l,\infty}, D_{j,\infty},D_{h,\infty}$. Thus one can choose a  neighborhood $O_q$ of $q$ such that $O_q, D_i$ are disjoint for every $v_i\in V\setminus\{v_l,v_j,v_h\}$. That means
\begin{equation}\label{E-3-7}
\big(\cup_{v_i\in A}D_{i,\infty}\big)\cap U_q\,=\,\emptyset,
\end{equation}
where $U_q=O_q\setminus(D_{l,\infty}\cup D_{j,\infty})$ is a non-empty set. Consequently,
\[
\cup_{v_i\in A} D_{i,\infty}\,\subset\,\hat{\mathbb C}\setminus U_q  \,\subsetneq\,\hat{\mathbb C}.
\]

\textbf{Step5}, $(\mathbf{z}_\infty,\mathbf{r}_\infty)\in M^\star_{IG}$. It remains to prove the marked triangle does not tend to a point, a line segment or a hemisphere, which follows from Lemma~\ref{L-2-2} and Lemma~\ref{L-2-3}.
\end{proof}

\begin{proposition}\label{P-3-7}
Suppose $[v_l,v_j], [v_j,v_i], [v_i,v_l]$ form a simple closed curve which separates the vertices of $\mathcal T$. Then
\[
D_{l,\infty}\cap D_{j,\infty}\cap D_{i,\infty}\,=\,\emptyset.
\]
\end{proposition}

\begin{proof}
First it is easy to see $\mathbb D_{l,\infty}\cup\mathbb D_{j,\infty}\cup \mathbb D_{i,\infty}\subsetneq\hat{\mathbb C}$. Otherwise, for $n_k$ sufficiently large, we have $\mathbb D_{l,n_k}\cup\mathbb D_{j,n_k}\cup \mathbb D_{i,n_k}=\hat{\mathbb C}$,
which contradicts that $\mathcal P_{n_k}$ is \textbf{irreducible}. Recall that
\[
\Theta_\infty([v_l,v_j])+\Theta_\infty([v_j,v_i])+\Theta_\infty([v_i,v_l])\,<\,\pi.
\]
By Lemma~\ref{L-2-7}, we prove the proposition.
\end{proof}

\begin{lemma}\label{L-3-8}
There exists an \textbf{irreducible} circle pattern $\mathcal P$ on the Riemann sphere $\hat{\mathbb C}$ realizing the data
$(\mathcal T,\widetilde{\Theta})$. Moreover, $\mathcal P$ is unique up to linear and anti-linear fractional maps.
\end{lemma}

\begin{proof}
In view of Andreev's Theorem and Thurston's observation, there exists a $\mathcal T$-type circle pattern $\mathcal P$ on
$\hat{\mathbb C}=\partial \mathbb H^3$ with exterior intersection angles given by $\widetilde{\Theta}$. Furthermore, using an appropriate isometry of $\mathbb H^3$, we may assume the origin $O\in \mathbb H^3$ is an interior point of the corresponding hyperbolic polyhedron. It is clear that each disk $D_i$ of $\mathcal P$ has radius less than $\pi/2$, otherwise $O$ does not belong to the half space for $D_i$, which is equal to the closed convex hull of ideal points in $\hat{\mathbb C}\setminus D_i\subset\partial \mathbb H^3$.

First we check that the contact graph of $\mathcal P$ is isomorphic to the $1$-skeleton of $\mathcal T$. For each pair of non-adjacent vertices $v_\alpha,v_\beta$, we need to show $D_\alpha\cap D_\beta=\emptyset$.
Recall that $S\big(\{v_\alpha\}\big)$ is the union of \textbf{open} $d$-simplices $(d=0,1,2)$ of $\mathcal T$ incident to $v_\alpha$. As in Figure~\ref{F-8}, we claim
\[
D_\alpha\,\subset\,S\big(\{v_\alpha\}\big).
\]
Assume $v_{1},\cdots,v_{m}\in V$, in anticlockwise order, are all adjacent vertices of $v_\alpha$. It suffices to prove the geodesic segment $\overline{z_\mu z_{\mu+1}}$ between the vertices $v_\mu$ and $v_{\mu+1}$ (set $v_{m+1}=v_1$) locates outside the disk $D_\alpha$ for $\mu=1,2,\cdots,m$. To this end, let us demonstrate
\begin{equation}\label{E-3-8}
d(z_\alpha,p)\,>\,r_\alpha
\end{equation}
for every point $p$ in $\overline{z_\mu z_{\mu+1}}$.

\begin{figure}[htbp]\centering
\includegraphics[width=0.54\textwidth]{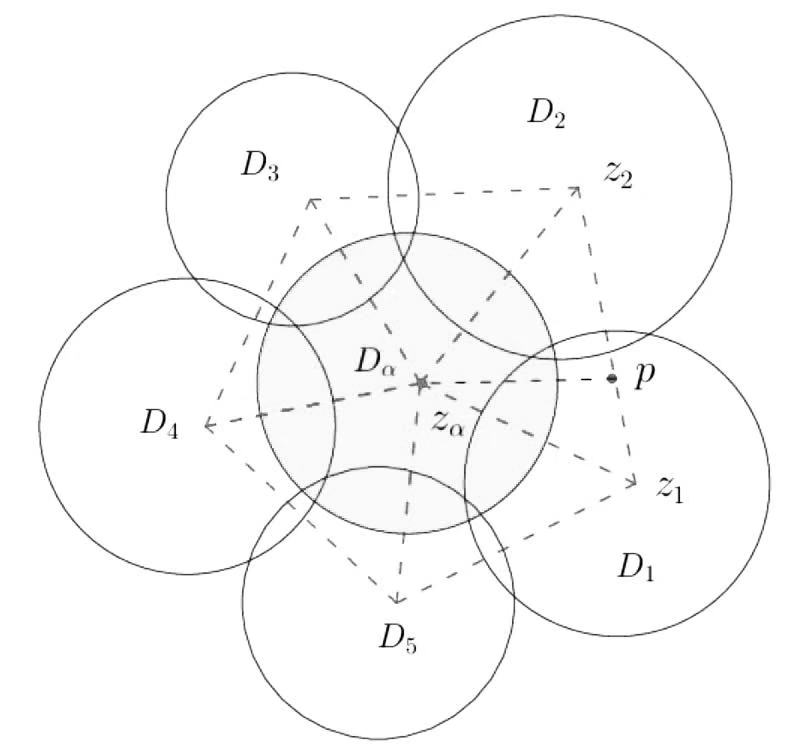}
\caption{A flower of disks with acute exterior intersection angles}\label{F-8}
\end{figure}

We start with the special case that $p=z_\mu$ or $p=z_{\mu+1}$. Since $r_\alpha, r_\mu ,\widetilde{\Theta}\big([v_\alpha,v_\mu]\big)\in(0,\pi/2)$, the formula
\[
\cos d(z_\alpha,z_\mu)\,=\,\cos r_\alpha\cos r_\mu-\cos\widetilde{\Theta}\big([v_\alpha,v_\mu]\big)\sin r_\alpha\sin r_\mu
\]
indicates $d(z_\alpha,z_\mu)>r_\alpha$. Similarly, $d(z_\alpha,z_{\mu+1})>r_\alpha$.
Let $x=d(z_\mu,p),$ $\theta=\angle z_\alpha z_\mu z_{\mu+1}$. Note that
\[
g(x)\,:=\,\cos d(z_\alpha,p)\,=\,\cos d(z_\alpha,z_\mu)\cos x+\cos\theta\sin d(z_\alpha,z_\mu)\sin x
\]
is a continuous function on the interval $\big[0,d(z_\mu,z_{\mu+1})\big]$. If  $g$ attains maximum at one of the endpoints, inequality~\eqref{E-3-8} naturally holds. Provided $g$ attains maximum at an interior point $x_0=d(z_\mu, p_0)$,   Fermat's Lemma gives
\[
0\,=\,g'(x_0)\,=\,-\cos d(z_\alpha,z_\mu)\sin x_0+\cos\theta\sin d(z_\alpha,z_\mu)\cos x_0.
\]
It follows that
\begin{equation}\label{E-3-9}
\cos d(z_\mu,p_0)\cos d(z_\alpha,p_0)\,=\,g(x_0)\cos x_0\,=\,\cos d(z_\alpha,z_\mu).
\end{equation}
Therefore, the geodesic segments $\overline{z_\alpha p_0}$ and $\overline{z_\mu z_{\mu+1}}$ are orthogonal at $p_0$. Set $y_0=d(p_0,z_{\mu+1})$. Because $x_0+y_0=d(z_\mu,z_{\mu+1})<r_\mu+r_{\mu+1}$, we have $x_0<r_\mu$ or $y_0<r_{\mu+1}$. Without loss of generality, assume the former case holds. That means
\[
\cos d(z_\alpha,z_\mu)\,=\,\cos r_\alpha\cos r_\mu-\cos\widetilde{\Theta}\big([v_\alpha,v_\mu]\big)\sin r_\alpha\sin r_\mu\,<\,\cos r_\mu\cos r_\alpha\,<\,\cos x_0\cos r_\alpha.
\]
Together with~\eqref{E-3-9}, we obtain
\[
\cos d(z_\alpha,p)\,=\,g(x)\,\leq\,g(x_0)\,< \,\cos r_\alpha,
\]
which also asserts~\eqref{E-3-8}. As a result,
\[
D_\alpha\,\subset\,S\big(\{v_\alpha\}\big).
\]
Hence
\[
D_\alpha\cap D_\beta\,\subset\,S\big(\{v_\alpha\}\big)\cap S\big(\{v_\beta\}\big)\,=\,\emptyset.
\]

To prove $\mathcal P$ is \textbf{irreducible}, we need to show $\cup_{v_i\in A}D_i\subsetneq\hat{\mathbb C}$ for every proper subset $A\subsetneq V$. Choosing a vertex $v_h\in V\setminus A$, we easily check that the center $z_h$ of the disk $D_h$ locates outside $D_i$ for every $v_i\in A\subset V\setminus\{v_h\}$. Thus
\[
\cup_{v_i\in A}D_i\,\subset\,\hat{\mathbb C}\setminus\{z_h\}\,\subsetneq\,\hat{\mathbb C}.
\]

In summary, $\mathcal P$ is the required circle pattern realizing the data $(\mathcal T,\widetilde{\Theta})$. Meanwhile, the rigidity part follows from the rigidity of the corresponding hyperbolic polyhedron, which is also implied by Andreev's Theorem.
\end{proof}

\begin{remark}
See also the work of Bowers-Stephenson \cite{Bowers-Stephenson} for the existence and rigidity of $\mathcal T$-type circle patterns with non-obtuse exterior intersection angles.
\end{remark}

Such a circle pattern can be mapped into a normalized one through a unique linear or anti-linear fractional map. With these preparations, we acquire the following result.
\begin{theorem}\label{T-3-10}
Let $\Lambda$ be as above. Then
\[
\deg(\mathcal{E}v,\Lambda,\Theta)\,=\,1
\;\;\text{\emph{or}}\;\;\deg(\mathcal{E}v,\Lambda,\Theta)\,=\,-1.
\]
\end{theorem}

\begin{proof}
We begin by computing $\deg(\mathcal{E}v,\Lambda,\widetilde{\Theta})$. Because of Lemma~\ref{L-3-8}, $\mathcal{E}v^{-1}(\widetilde{\Theta})$ consists of a unique point. Noting that $\widetilde{\Theta}$ is a regular value of the map $\mathcal{E}v$, we get
\[
\deg(\mathcal {E}v,\Lambda,\widetilde{\Theta})\,=\,1 \quad \text{or} \quad \deg(\mathcal{E}v,\Lambda,\widetilde{\Theta})\,=\,-1.
\]
In light of Lemma~\ref{L-3-4} and Theorem~\ref{T-6-8}, the assertion is proved.
\end{proof}

\subsection{Existence and rigidity}

Applying Theorem~\ref{T-3-10}, we immediately derive the main result of this section.

\begin{proof}[\textbf{Proof of Theorem~\ref{T-3-1}}]
By Theorem~\ref{T-3-10} and Theorem~\ref{T-6-9},  any function
$\Theta: E\to(0,\pi)$ satisfying conditions   $\mathrm{\mathbf{(c1)}}- \mathrm{\mathbf{(c4)}}$ and $\mathrm{\mathbf{(m5)}}$ is in the image of the map $\mathcal{E}v$. That means there exists an irreducible circle pattern $\mathcal P$ on $\hat{\mathbb C}$ realizing the data $(\mathcal T, \Theta)$.
\end{proof}

Meanwhile, Sard's Theorem indicates almost all these patterns are local rigidity.
\begin{theorem}\label{T-3-11}
For almost every $\Theta\in W_{m}$, up to linear and anti-linear fractional maps, there are at most finitely many \textbf{irreducible} circle patterns on $\hat{\mathbb C}$ realizing the data
$(\mathcal T,\Theta)$.
\end{theorem}

\begin{proof}
Set $W_0=\mathcal Ev(M_{0})$, where $M_{0}\subset M_{IG}^\star$ denotes the set of critical points of the map $\mathcal Ev$. By Sard's Theorem, $M_0$ has zero measure. What is more, the boundary set $\partial W_{m}$ has also zero measure. Hence $W_{m}\setminus (W_0\cup\partial W_{m})$ is a subset of $W_{m}$ with full measure. Recall that
\[
\dim(M^\star_{IG})\,=\,\dim (W_m)\,=\,|E|.
\]
For $\Theta\in W_{m}\setminus(W_0\cup\partial W_{m})$, the Regular Value Theorem (Theorem~\ref{T-6-1}) implies
$\mathcal Ev^{-1}(\Theta)$ is a discrete set. Moreover, in a way similar to the proof of Lemma \ref{L-3-4},  we show this set is compact. Thus
$\mathcal Ev^{-1}(\Theta)$ consists of finite points, which completes the proof.
\end{proof}

We are led to the generalization of Andreev's Theorem.
\begin{proof}[\textbf{Proof of Theorem~\ref{T-1-8}}]
By Theorem~\ref{T-3-1}, there exists an irreducible circle pattern $\mathcal P$ on $\hat{\mathbb C}$ realizing the dual data $(P^\ast,\Theta^\ast)$. For each disk $D_i$ of $\mathcal P$, let $H_i\subset \mathbb H^3$ be the closed convex hull of ideal points in $\hat{\mathbb C}\setminus D_i\subset\partial\mathbb H^3$. Evidently, each $H_i$ is a half space. It suffices to verify that $Q=\cap_{i=1}^{|V|} H_i$ is a compact convex polyhedron realizing the data $(P,\Theta)$.

First let us show $Q$ has a non-empty interior. Suppose $v_l,v_j,v_h$ are the vertices of a triangle of $P^\ast$. For each
$v_i\in V\setminus\{v_l,v_j,v_h\}$, proceeding as in the proof of~\eqref{E-3-6}, we obtain
\begin{equation}\label{E-3-10}
D_l\cap D_j\cap D_h\cap D_i\,=\,\emptyset.
\end{equation}
Hence there is a neighborhood $B_q$ of the point $q=\partial_{\mathbb H^3} H_l\cap \partial_{\mathbb H^3} H_j\cap\partial_{\mathbb H^3} H_h$ such that $B_q\subset H_{i}$ for every
$v_i\in V\setminus\{v_l,v_j,v_h\}$, which implies
\[
H_l\cap H_j\cap H_h \cap B_q\,\subset\,Q.
\]
Evidently, $Q$ has a non-empty interior and is consequently a convex polyhedron.

Next we verify that $Q,P$ are combinatorially equivalent by asserting there exists a homeomorphism between their boundaries. Note that
\begin{equation}\label{E-3-11}
\partial_{\mathbb H^3} Q\,=\,Q\cap \overline{Q^c}\,=\,\cup_{i=1}^{|V|} Q\cap \overline{H_i^c}\,=\,\cup_{i=1}^{|V|} Q\cap H_i\cap \overline{H_i^c}\,=\,\cup_{i=1}^{|V|}Q\cap \partial_{\mathbb H^3} H_i.
\end{equation}
Suppose the $i$-th face of $P$ is $m$-sided. It suffices to prove the $i$-th face
\[
\mathcal F_i\,:=\,Q\cap \partial_{\mathbb H^3} H_i\,=\,\big(\cap_{\mu=1}^{|V|}H_\mu\big)\cap\partial_{\mathbb H^3} H_i.
\]
of $Q$ is also $m$-sided. For any $v_\mu\in V\setminus\{v_i\}$
which is not adjacent to $v_i$, we claim
\[
\partial_{\mathbb H^3} H_i\,\subset\,H_\mu.
\]
Otherwise, $D_i\cap D_\mu\neq \emptyset$, which contradicts that $D_i, D_\mu$ are disjoint. Assume $v_1,\cdots,v_m$, in anticlockwise order, are all adjacent vertices of $v_i$. Then
\begin{equation}\label{E-3-12}
\mathcal F_i\,=\,\big(\cap_{\mu=1}^{|V|}H_\mu\big)\cap\partial_{\mathbb H^3} H_i\,=\,\big(\cap_{\mu=1}^m H_\mu\big)\cap\partial_{\mathbb H^3} H_i.
\end{equation}
By abuse of notation, let $\partial \mathcal F_i$ represent the boundary of $\mathcal F_i$ in the geodesic plane $\partial_{\mathbb H^3} H_i$. Similar reasoning  to the proofs of~\eqref{E-3-11} and~\eqref{E-3-12} yields
\[
\partial\mathcal F_i\,=\,\cup_{\mu=1}^m \mathcal F_i\cap \partial_{\mathbb H^3} H_\mu\,=\,\cup_{\mu=1}^m H_{\mu-1}\cap H_{\mu+1}\cap\partial_{\mathbb H^3} H_i\cap\partial_{\mathbb H^3} H_\mu\,:=\,\cup_{\mu=1}^m \mathcal L_{\mu}.
\]
Here we set $H_0=H_m$ and $H_{m+1}=H_1$. It is easy to see $\mathcal L_{\mu}$ is a geodesic segment with endpoints $q_{\mu},q_{\mu+1}$, where
 $q_\mu=\partial_{\mathbb H^3} H_{\mu-1}\cap\partial_{\mathbb H^3} H_{\mu}\cap\partial_{\mathbb H^3} H_i$. Thus $\partial \mathcal F_i$ is a simple closed curve consisting of $m$ geodesic segments. Namely, $\mathcal F_i$ is an $m$-sided polygon.

In summary, $Q$ is a compact convex hyperbolic polyhedron realizing the data $(P,\Theta)$. Finally, the rigidity is a consequence of the following Theorem~\ref{T-3-12}.
\end{proof}

The theorem below was proved by Rivin-Hodgson~\cite[Corollary 4.6]{Rivin-Hodgson} via analogous arguments used by Cauchy in the proof of his rigidity theorem~\cite[Chap. 12]{Aigner-Ziegler} for compact convex polyhedra in Euclidean 3-space. Moreover, a similar proof can be also seen in the work of Roeder-Hubbard-Dunbar~\cite[Proposition 4.1]{Roeder-Hubbard-Dunbar}. In fact, this result greatly generalizes Theorem~\ref{T-3-11} once we relate hyperbolic polyhedra to circle patterns.

\begin{theorem}[Rivin-Hodgson]\label{T-3-12}
Compact convex polyhedra in hyperbolic 3-space $\mathbb H^3$ with trivalent vertices are determined up to congruences by their combinatorics and dihedral angles.
\end{theorem}

\begin{remark}
A more direct approach to Theorem~\ref{T-1-8} is to follow the continuity method used in the works of Aleksandrov~\cite{Aleksandrov}, Andreev~\cite{Andreev}, Rivin-Hodgson~\cite{Rivin-Hodgson}, Rivin~\cite{Rivin} and Roeder-Hubbard-Dunbar~\cite{Roeder-Hubbard-Dunbar}. In particular, most arguments in~\cite{Andreev,Roeder-Hubbard-Dunbar} remain valid in our situation. The only substantial modification might be showing the property that all faces of the limit polyhedron are still parallelograms. Toward this end, Andreev~\cite{Andreev} and Roeder-Hubbard-Dunbar~\cite{Roeder-Hubbard-Dunbar} made the use of Gauss-Bonnet Formula, while we could proceed as in the third step of the proof of Lemma~\ref{L-3-4}.
\end{remark}

\section{Patterns of circles with interstices}\label{S-4}
It remains to consider those circle patterns possessing at least one interstice. Given such a circle pattern $\mathcal P$, using an appropriate M\"{o}bius transformation, we assume the infinity point belongs to an interstice. In this way $\mathcal P$ is regarded as a circle pattern on the complex plane $\mathbb C=\hat{\mathbb C}\setminus\{\infty\}$. We will prove the following result which extends the existence part of Marden-Rodin Theorem.

\begin{theorem}\label{T-4-1}
Let $\mathcal T$ be a triangulation of the sphere and let $\Theta:E\to [0,\pi)$ be a function satisfying conditions $\mathrm{\mathbf{(c1)}},\mathrm{\mathbf{(c2)}},
\mathrm{\mathbf{(c3)}},\mathrm{\mathbf{(c4)}}$
and the condition below:
\begin{itemize}
\item[$\mathrm{\mathbf{(g5)}}$] There exists a triangle of $\mathcal T$ with edges $e_a,e_b,e_c$ such that $\sum_{\mu=a,b,c}\Theta(e_\mu)<\pi$.
\end{itemize}
Then there exists an \textbf{irreducible} circle pattern $\mathcal P$ on the Riemann sphere $\hat{\mathbb C}$ with contact graph isomorphic to the $1$-skeleton of $\mathcal T$ and exterior intersection angles given by $\Theta$.
\end{theorem}

On this occasion, we endow $\mathbb C$ with the Riemannian metric $\mathrm{d s}=|\mathrm{d}z|$. Then
$(\mathbb C,\mathrm{d s})$ is isometric to the Euclidean plane $\mathbb R^2$.

\subsection{Configuration spaces}
Recall that $V,E,F$ are the sets of vertices, edges and triangles of $\mathcal T$. Let $\triangle^\infty$ be the triangle with edges $e_a,e_b,e_c$ and let $v_a,v_b,v_c$ be the vertices opposite to $e_a,e_b,e_c$, respectively. For simplicity, we assume  the infinity point belongs to $\triangle^\infty$. Therefore, $\mathcal T'=\mathcal T\setminus\{\triangle^\infty\}$ is a triangulation of the region $\hat{\mathbb C}\setminus \triangle^\infty\subset\mathbb C$.

Set $Z=\mathbb C^{|V|}\times \mathbb R_{+}^{|V|}$. Note that $Z$ is a smooth manifold of dimension $|E|+6$ which parameterizes the space of patterns of $|V|$ disks on the complex plane $(\mathbb C,\mathrm{ds})$. As before, we call each point $(\mathbf{z},\mathbf{r})=(z_1,\cdots,z_{|V|},r_1,\cdots,r_{|V|})\in Z$  a configuration, since it gives a circle pattern
$\mathcal P=\{D_i\}_{i=1}^{|V|}$ on  $(\mathbb C,\mathrm{d s})$, where the disk $D_i$ is centered at $z_i$ and is of radius $r_i$.

For each edge $e=[v_i,v_j]\in E$, the inversive distance is defined as
\[
I(e,\mathbf{z},\mathbf{r})\,=\,\frac{|z_i-z_j|^2 -r_i^2-r_j^2}{2r_i r_j}.
\]
The subspace $Z_{E}\subset Z$ is the set of configurations subject to
\[
-1\,<\,I(e,\mathbf{z},\mathbf{r})\,<\,1
\]
for every $e\in E$. We are interested in the subspace $Z_{\mathcal T}\subset Z_{E}$ of configurations which give $\mathcal T'$-type circle patterns. More precisely, $(\mathbf{z},\mathbf{r})\in Z_{\mathcal T}$ if and only if there exists a geodesic triangulation
$\mathcal T'(\mathbf{z},\mathbf{r})$ embedding into
$(\mathbb C,\mathrm{d s})$ with the following properties:
\begin{itemize}
\item[$\mathrm{\langle \mathbf{y1}\rangle}$] $\mathcal T'(\mathbf{z},\mathbf{r})$ is isotopic to $\mathcal T'$;
\item[$\mathrm{\langle \mathbf{y2}\rangle}$]The vertices of $\mathcal T'(\mathbf{z},\mathbf{r})$ coincide with the centers $z_1,z_2,\cdots,z_{|V|}$.
\end{itemize}
Let $Z_{G}\subset Z_{\mathcal T}$  consist of configurations under the further restriction:
\begin{itemize}
\item[$\mathrm{\langle \mathbf{y3}\rangle}$] The disks
$D_\alpha, D_\beta$ are disjoint whenever there is no edge between $v_\alpha$ and $v_\beta$.
\end{itemize}
Note that the subspaces $Z_{E},Z_{\mathcal T}, Z_G$ are all open in $Z$ and thus are smooth manifolds of dimension $|E|+6$.

\begin{figure}[htbp]\centering
\includegraphics[width=0.49\textwidth]{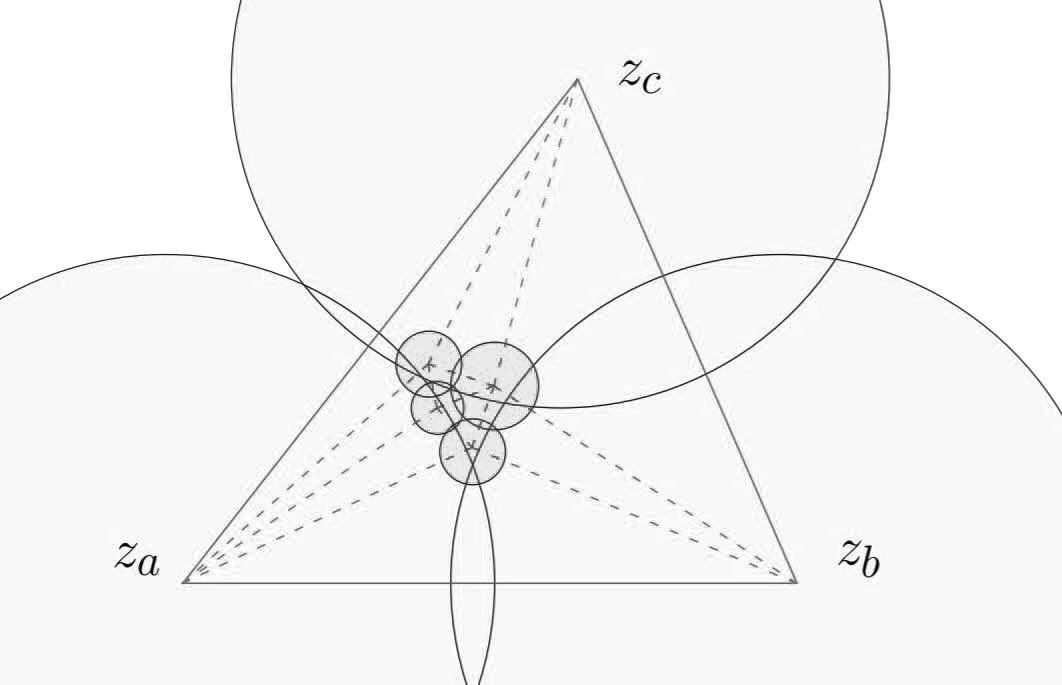}
\caption{A normalized planar circle pattern}\label{F-9}
\end{figure}

In the end we use $Z^\star_{G}\subset Z_{G}$ to represent the subspace of configurations such that the following normalization conditions hold:
\begin{itemize}
\item[$\mathrm{\langle \mathbf{y4}\rangle}$] $z_a\,=\,0$,\;\, $z_b\,>\,0$,\;\,$0\,<\,\operatorname{Arg}(z_c)\,<\,\pi$; \item[$\mathrm{\langle \mathbf{y5}\rangle}$]$r_a\,=\,r_b\,=\,r_c$;
\item[$\mathrm{\langle \mathbf{y6}\rangle}$] $\sum\nolimits_{i=1}^{|V|}r_i\,=\,1$.
\end{itemize}
It is easy to see $Z^\star_{G}$ is a smooth manifold of dimension $|E|$. Meanwhile, we have the following smooth map
\[
\begin{aligned}
\mathcal{R}v:\quad &\;Z^\star_{G} &\longrightarrow \quad &\quad\;  Y:=(0,\pi)^{|E|}\\
&(\mathbf{z},\mathbf{r})  &\longmapsto \quad  & \big(\Theta(e_1,\cdot),\Theta(e_2,\cdot),\cdots\big),
\end{aligned}
\]
where $\Theta(e,\cdot)=\arccos I(e,\cdot)$.

\begin{lemma}\label{L-4-2}
Suppose $\mathcal P$ is a $\mathcal T'$-type circle pattern on
$(\mathbb C,\mathrm{d s})$ whose exterior intersection angle function $\Theta: E\to[0,\pi)$ satisfies conditions $\mathrm{\mathbf{(c1)}}$ and $\mathrm{\mathbf{(g5)}}$. Let $v_\alpha,v_\beta\in V$ be a pair of non-adjacent vertices. For any
$p\in D_{\alpha}\cap D_{\beta}$, there exists
$v_\eta\in V\setminus\{v_\alpha,v_\beta\}$ such that $p\in\mathbb D_{\eta}$.
\end{lemma}

\begin{proof}
We treat the region $\mathbb C\setminus \mathcal T'$ as a triangle with vertices $v_a,v_b,v_c$ and define  $S\big(\{v_\alpha\}\big),$ $ \operatorname{\mathbb{FL}_\alpha}$ as Lemma~\ref{L-3-2}. Similarly, we want to show
\[
D_\alpha\,\subset\,\operatorname{\mathbb{FL}_\alpha}.
\]
To bypass complexities that $S\big(\{v_\alpha\}\big)$ is unbounded or $\partial S\big(\{v_\alpha\}\big)\nsubseteq {\mathbb{FL}_\alpha}$ (when $\Theta(e)=0$), we choose a closed Jordan region $\Upsilon_\alpha\subset S\big(\{v_\alpha\}\big)$. Note that
\[
\partial \Upsilon_\alpha\,\subset\, \Upsilon_\alpha\, \subset\, S\big(\{v_\alpha\}\big)\,\subset\, \operatorname{\mathbb{FL}}_\alpha.
\]
We can construct a family of closed Jordan regions $\{J_t\}_{0\leq t\leq 1}$ satisfying
\[
J_0\,=\,\Upsilon_\alpha,\quad\, J_1\,=\,D_\alpha, \quad\,\partial J_t\,\subset\, \operatorname{\mathbb{FL}}_\alpha.
\]
The remainder part of the proof is parallel to that of Lemma~\ref{L-3-2}. We omit the details.
\end{proof}

\subsection{Topological degree}
First we deal with the case that $\Theta(e)>0$ for every $e\in E$. For any function $\Theta: E\to(0,\pi)$ which satisfies conditions $\mathrm{\mathbf{(c1)}}-\mathrm{\mathbf{(c4)}}$ and $\mathrm{\mathbf{(g5)}}$, it suffices to show it is in the image of the map $\mathcal{R}v$. Recall that
\[
\dim(Z^\star_{G})\,=\,\dim (Y)\,=\,|E|.
\]
To this end, we still make the use of topological degree theory. Specifically, we need to find a relatively compact open set
$\Omega\subset Z^\star_{G}$ and determine the degree $\deg(\mathcal{R}v,\Omega,\Theta)$.

Following the strategy in last section, we deform $\Theta$ to another value such that the degree is relatively easier to manipulate.
Let $W_{g}$ denote the set of functions satisfying conditions $\mathrm{\mathbf{(c1)}}-\mathrm{\mathbf{(c4)}}$ and $\mathrm{\mathbf{(g5)}}$. By Sard's Theorem, there exists at least one regular value $\Theta^\diamond\in(0,\pi/4)^{|E|}\subset W_g$ of the map $\mathcal{R}v$. For $t\in[0,1]$, set
 \[
 \Theta(t)\,=\,(1-t)\Theta^\diamond+t\Theta.
 \]
Note that $\big\{\Theta(t)\big\}_{0\leq t\leq 1}$ form a continuous curve $\gamma$ in $W_g$. We need the following result.

\begin{lemma}\label{L-4-3}
There exists a relatively compact open subset $\Omega\subset Z^\star_{G}$ such that
\[
\mathcal{R}v^{-1}(\gamma)\,\subset\,\Omega.
\]
\end{lemma}

As before, it suffices to show any sequence  $\big\{(\mathbf{z}_n,\mathbf{r}_n)\big\}\subset \mathcal{R}v^{-1}(\gamma)$ contains a convergent subsequence in $Z^\star_{G}$. Note that each configuration $(\mathbf{z}_n,\mathbf{r}_n)$ gives a normalized circle pattern $\mathcal P_n=\{D_{i,n}\}_{i=1}^{|V|}$ embedding into $(\mathbb C,\mathrm{ds})$ and realizing the data $(\mathcal T,\Theta_n)$, where
\[
\Theta_n\,=\,(1-t_n)\Theta^\diamond+t_n\Theta
\]
for some $t_n\in[0,1]$. For every $v_i\in V\setminus\{v_a,v_b,v_c\}$, the apex curvature $K_{i,n}$ is defined as
\[
K_{i,n}\,=\,2\pi-\sigma_{i,n},
\]
where $\sigma_{i,n}$ denotes the sum of inner angles at $v_i$ for all triangles of $\mathcal T'(\mathbf{z}_n,\mathbf{r}_n)$ incident to the vertex $v_i$. Obviously we have
\begin{equation}\label{E-4-1}
K_{i,n}\,=\,0.
\end{equation}

\begin{proposition}\label{P-4-4}
There is a subsequence of $\big\{(\mathbf{z}_{n},\mathbf{r}_{n})\big\}$  converging to a point $(\mathbf{z}_\infty,\mathbf{r}_\infty)\in Z$.
\end{proposition}

\begin{proof}
We reduce the proof to verifying the following two properties:
\begin{itemize}
\item[$(i)$] For $i=1,2,\cdots,|V|$, the radius $r_{i,n}$ is bounded from below and above by positive constants;
\item[$(ii)$] For $i=1,2,\cdots,|V|$, the coordinate $z_{i,n}$ is bounded.
\end{itemize}

We begin with the first property. Under normalization condition $\mathrm{\langle \mathbf{y6}\rangle}$, it is easy to see  each $r_{i,n}$ is bounded from
 above by $1$. In addition, there is a subsequence $\{\mathbf{r}_{n_k}\}$ converging to some vector $\mathbf{r}_\infty\in [0,1]^{|V|}$. Next we show $r_{i,n}$ is bounded from below by a positive constant. Assume it is not true. Then there exists $v_{i_0}\in V$ satisfying
\[
r_{i_0,n_k}\,\to\,0.
\]
Without loss of generality, suppose $t_{n_k}$ converges to $t_\infty\in[0,1]$. Otherwise, one extracts a convergent subsequence to achieve the goal. Let $V_0\subset V$ denote the set of vertices $v_i\in V$ for which $r_{i,n_k}\to 0$. Due to normalization condition
$\mathrm{\langle \mathbf{y5}\rangle}$, we have
\[
v_a,v_b,v_c\,\in\, V_0\quad\, \text{or}\quad\, \{v_a,v_b,v_c\}\cap V_0\,=\,\emptyset.
\]
The first case never occurs. Otherwise, every triangle of
$\mathcal T'(\mathbf{z}_{n_k},\mathbf{r}_{n_k})$ degenerates to a point, which implies every radius $r_{i,n_k}$ tends to zero. In view of condition $\mathrm{\langle \mathbf{y6}\rangle}$, this leads to a contradiction. We thus deduce that $V_0$ is a non-empty subset of $V\setminus\{v_a,v_b,v_c\}$.

Furthermore, by Lemma~\ref{L-2-5}, similar reasoning to Proposition~\ref{P-3-6} yields
\[
\sum\nolimits_{v_i\in V_0}K_{i,n_k}\,\to\, -\sum\nolimits_{(e,u)\in Lk(V_0)}\big[\pi-(1-t_\infty)\Theta^\diamond(e)-t_\infty\Theta(e)\big]+2\pi\chi\big(S(V_0)\big).
\]
Together with~\eqref{E-4-1}, this gives
\begin{equation}\label{E-4-2}
0\,=\,-\sum\nolimits_{(e,u)\in Lk(V_0)}\big[\pi-(1-t_\infty)\Theta^\diamond(e)-t_\infty\Theta(e)\big]+2\pi\chi\big(S(V_0)\big).
\end{equation}
Suppose the triangulation $\mathcal T$ possesses more than four vertices. Proceeding as in the proof of Proposition~\ref{P-3-5}, we prove that~\eqref{E-4-2} leads to contradictions. Provided $\mathcal T$ possesses exactly four vertices, then $\mathcal T$ is the boundary of a tetrahedra and $V_0$ consists of only one vertex. What is more, we reduce~\eqref{E-4-2} to
\[
\sum\nolimits_{\mu=a,b,c}(1-t_\infty)\Theta^\diamond(e_\mu)+\sum\nolimits_{\mu=a,b,c}t_\infty\Theta(e_\mu)\,=\,\pi.
\]
Recalling that $0<\Theta^\diamond<\pi/4$, we get $\sum\nolimits_{\mu=a,b,c}\Theta(e_\mu)\geq\pi$, which contradicts condition $\mathrm{\mathbf{(g5)}}$.

It remains to check the second property. Because every $r_{i,n}$ is bounded, every triangle of $\mathcal T'(\mathbf{z}_{n},\mathbf{r}_{n})$ is of finite size. Meanwhile, by condition
$\mathrm{\langle \mathbf{y4}\rangle}$, the coordinate $z_{a,n}$ is fixed. Putting these relations together, we conclude that every $z_{i,n}$ is bounded.
\end{proof}

Now it is ready to demonstrate Lemma~\ref{L-4-3}.
\begin{proof}[\textbf{Proof of Lemma~\ref{L-4-3}}]
Proposition \ref{P-4-4} indicates there exists a subsequence $\big\{(\mathbf{z}_{n_k},\mathbf{r}_{n_k})\big\}$  such that
\[
(\mathbf{z}_{n_k},\mathbf{r}_{n_k})
\,\to\,(\mathbf{z}_\infty,\mathbf{r}_\infty)\,\in\,Z.
\]
Suppose $\mathcal P_\infty=\{D_{i,\infty}\}_{i=1}^{|V|}$ is the circle pattern given by $(\mathbf{z}_\infty,\mathbf{r}_\infty)$ and  $\Theta_\infty$ is the exterior intersection angle function of $\mathcal P_\infty$. We divide the proof into the following steps:

\textbf{Step1}, $(\mathbf{z}_\infty,\mathbf{r}_\infty)\in Z_E$. Since $\Theta_\infty=(1-t_\infty)\Theta^\diamond+t_\infty\Theta$, the assertion is straightforward.

\textbf{Step2}, $(\mathbf{z}_\infty,\mathbf{r}_\infty)\in Z_{\mathcal T}$. By Lemma~\ref{L-2-4}, the statement follows verbatim from the second step of the proof of Lemma~\ref{L-3-4}.

\textbf{Step3}, $(\mathbf{z}_\infty,\mathbf{r}_\infty)\in Z_{G}$. For each pair of non-adjacent vertices $v_\alpha,v_\beta$, we need to show the disks $D_{\alpha,\infty}, D_{\beta,\infty}$ are disjoint. Assume on the contrary that $D_{\alpha,\infty}\cap D_{\beta,\infty}\neq\emptyset$. Then either $D_{\alpha,\infty}\cap D_{\beta,\infty}$ contains interior points or $D_{\alpha,\infty}\cap D_{\beta,\infty}$ consists of a single point. In light of Lemma~\ref{L-4-2} and Lemma~\ref{L-2-6}, similar arguments to the third step of the proof of Lemma~\ref{L-3-4} lead to contradictions.

\textbf{Step4}, $(\mathbf{z}_\infty,\mathbf{r}_\infty)\in Z_{G}^\star$. It is an immediate consequence of Lemma~\ref{L-2-4}.
\end{proof}

As a result, we determine the degree of the map $\mathcal{R}v$.
\begin{theorem}\label{T-4-5}
Let $\Omega$ be as above. Then
\[
\deg(\mathcal{R}v,\Omega,\Theta)\,=\,1
\;\;\text{\emph{or}}\;\;\deg(\mathcal{R}v,\Omega,\Theta)\,=\,-1.
\]
\end{theorem}
\begin{proof}
It suffices to compute $\deg(\mathcal{R}v,\Omega,\Theta^\diamond)$. Theorem~\ref{T-1-1} implies $\mathcal{R}v^{-1}(\Theta^\diamond)$ consists of a unique point. Because $\Theta^\diamond$ is a regular value of the map $\mathcal{R}v$, we get
\[
\deg(\mathcal{R}v,\Omega,\Theta^\diamond)\,=\,1 \quad \text{or} \quad \deg(\mathcal{R}v,\Omega,\Theta^\diamond)\,=\,-1.
\]
Using Lemma~\ref{L-4-3} and Theorem~\ref{T-6-8}, we finish the proof.
\end{proof}

\subsection{Main results}
We are led to the general Marden-Rodin Theorem.
\begin{proof}[\textbf{Proof of Theorem~\ref{T-4-1}}]
Provided $\Theta(e)>0$ for every $e\in E$,  it is easy to derive the result from  Theorem~\ref{T-4-5} and Theorem~\ref{T-6-9}. Suppose $\Theta(e)=0$ for some $e\in E$. Choose $\epsilon>0$ such that $\Theta_\epsilon=\Theta+\epsilon$ satisfies conditions $\mathrm{\mathbf{(c1)}}-\mathrm{\mathbf{(c4)}}$ and $\mathrm{\mathbf{(g5)}}$. Then there exists a normalized \textbf{irreducible} circle pattern $\mathcal P_\epsilon$  realizing the data
$(\mathcal T,\Theta_\epsilon)$. As $\epsilon\to 0$, similar reasoning to Proposition~\ref{P-4-4} implies there is a subsequence of
$\{\mathcal P_\epsilon\}$ convergent to a circle pattern $\mathcal P_0$. In a way similar to the proof of Lemma~\ref{L-4-3}, we check that $\mathcal P_0$ is the required circle pattern realizing the data $(\mathcal T,\Theta)$.
\end{proof}

Applying Sard's Theorem again, we still have the local rigidity property.
\begin{theorem}\label{T-4-6}
For almost every $\Theta\in W_{g}$, up to linear and anti-linear fractional maps, there are at most finitely many \textbf{irreducible} circle patterns on $\hat{\mathbb C}$ realizing the data $(\mathcal T,\Theta)$.
\end{theorem}

\begin{proof}
The proof is the same as that of Theorem~\ref{T-3-11} if we can verify that every \textbf{irreducible} circle pattern $\mathcal P$ realizing the data $(\mathcal T,\Theta)$ has an interstice. In other words, we reduce the proof to showing $\cup_{v_i\in V}D_i\subsetneq\hat{\mathbb C}$. Remember that $v_a,v_b,v_c$ are the vertices opposite to the edges $e_a,e_b,e_c$. Let $\triangle_{abc}$ be the triangle whose vertices are the centers of $D_a,D_b,D_c$. Under condition $\mathrm{\mathbf{(g5)}}$, a similar argument to the proof of~\eqref{E-3-7} indicates there exists a neighborhood $O_q$ of the point $q=\partial D_a\cap \partial D_b\cap \triangle_{abc}$ such that
\[
\big(\cup_{v_i\in V}D_i\big)\cap U_q\,=\,\emptyset,
\]
where $U_q=O_q\setminus(D_a\cup D_b)$ is a non-empty set. It follows that
\[
\cup_{v_i\in V}D_i\,\subset\,\hat{\mathbb C}\setminus U_q\,\subsetneq\,\hat{\mathbb C},
\]
which concludes the theorem.
\end{proof}

Finally let us prove the main results.
\begin{proof}[\textbf{Proof of Theorem~\ref{T-1-4}}]
Combining Theorem~\ref{T-3-1} and Theorem~\ref{T-4-1} gives the assertion.
\end{proof}

\begin{proof}[\textbf{Proof of Theorem~\ref{T-1-6}}]
By Theorem~\ref{T-3-11} and Theorem~\ref{T-4-6}, the statement holds.
\end{proof}

\section{Some questions}\label{S-5}
The paper leaves open a number of questions. We think the following are some particularly interesting subjects for further developments.

The first question concerns a further generalization of Andreev's Theorem. Given an abstract trivalent polyhedron $P$, we call a simple closed curve $\Gamma$ formed by edges of the dual complex $P^\ast$ a \textbf{Whitehead circuit} if $\Gamma$ is the boundary of the union of two adjacent triangles, and we say a \textbf{Whitehead circuit} $\Gamma$ is \textbf{essential} if $\Gamma$ is at the same time the union of two \textbf{homologically non-adjacent} arcs. Specifically, we ask whether the following conjecture holds.
\begin{conjecture}
Let $P$ be an abstract trivalent polyhedron with more than four faces. Assume that $\Theta:E\to (0,\pi)$ is a function satisfying the  conditions below:
\begin{itemize}
\item[$\mathbf{(a1)}$] Whenever three distinct edges $e_1,e_2,e_3$  meet at a vertex, then $\sum_{\mu=1}^3\Theta(e_\mu)>\pi$,  and
    $\Theta(e_1)+\Theta(e_2)<\Theta(e_3)+\pi$, $\Theta(e_2)+\Theta(e_3)<\Theta(e_1)+\pi$, $\Theta(e_3)+\Theta(e_1)<\Theta(e_2)+\pi$.
\item[$\mathbf{(a2)}$] Whenever $\Gamma$ is an \textbf{essential Whitehead circuit} intersecting edges $e_1,e_2,e_3,e_4$, then $\sum_{\mu=1}^4\Theta(e_\mu)\leq 2\pi$, and one of the inequalities is strict if $P$ is the triangular prism.
\item[$\mathbf{(a3)}$] Whenever $\Gamma$ is a \textbf{prismatic k-circuit} intersecting edges $e_1,e_2,\cdots,e_k$, then $\sum_{\mu=1}^k\Theta(e_\mu)<(k-2)\pi$.
\end{itemize}
Then there exists a compact convex hyperbolic polyhedron $Q$ combinatorially equivalent to $P$ with dihedral angles given by $\Theta$. Furthermore, $Q$ is unique up to isometries of $\mathbb H^3$.
\end{conjecture}

We believe it a hopeful conjecture in light of the following reasons:
\begin{itemize}
\item{} Theorem~\ref{T-1-8} is a special case of the above conjecture.
\item{} The rigidity part has been proved by Rivin-Hodgson~\cite[Corollary 4.6]{Rivin-Hodgson}.
\item{} A key step of the proof is to show under suitable conditions a sequence of compact convex hyperbolic polyhedra $\{Q_n\}$ can produce a limit polyhedron $Q_\infty$. Condition $\mathbf{(a1)}$ implies the face angles of $Q_\infty$ do not degenerate. Meanwhile, condition $\mathbf{(a2)}$ prevents Whitehead moves. That means the combinatorics of $Q_\infty$ may not change.
\item{} Under conditions $\mathbf{(a1)}-\mathbf{(a3)}$, an analogous result to Lemma~\ref{L-2-1} holds. Following the proof of Proposition~\ref{P-3-5}, we can establish some non-degeneracy properties of the corresponding limit circle pattern.
\end{itemize}

We may also pursue the question of generalizing Theorem~\ref{T-1-4}. Nonetheless, it is much more difficult to control the combinatorial structure of the contact graph once the conditions of Theorem~\ref{T-1-4} are further relaxed. In fact, using Rivin's Theorem~\cite{Rivin}, we find that convex hyperbolic polyhedral in a fixed combinatorial class may provide circle patterns with non-isomorphic contact graphs.

Another question is to develop algorithms to search the desired circle patterns. In non-obtuse angle cases, there are already many successful computational methods in the works of Collins-Stephenson~\cite{Collins-Stephenson}, Chow-Luo~\cite{Chow-Luo}, Orick-Stephenson-Collins~\cite{Orick-Stephenson-Collins}, Connelly-Gortler~\cite{Connelly-Gortler}, Bowers~\cite{Bowers} and others.
We are looking forward to more numerical experiments to check whether these algorithms still work well under conditions of Theorem~\ref{T-1-4}. Moreover, can one prove the convergence?

Finally we are interested in the question whether it is possible to establish the global rigidity property under conditions of Theorem~\ref{T-1-4}. It is known that the answer is positive under extra condition $\mathbf{(m5)}$ or $\mathbf{(r1)}$ (see Theorem~\ref{T-1-8} and Remark~\ref{R-1-7}), while the general case remains open.


\section{Appendix}\label{S-6}
This section is devoted to some results from differential topology. Remind that similar contents to this appendix also appeared in~\cite{Liu-Zhou,Zhou,Jiang-Luo-Zhou}. Here we include them for the sake of completeness and we refer the readers to~\cite{Guillemin-Pollack, Hirsch,Milnor} for a fuller treatment.

Assume that $M, N$ are smooth manifolds of dimensions $m,n$. A point $x\in M$ is called a critical point of a $C^1$ map $f: M\to N$ if the tangent map  $df_x: T_x M\to T_{f(x)}N$ is not surjective. We use $C_f$  to denote the set of critical points of $f$ and define $N\setminus f(C_f)$ to be the set of regular values of $f$.

\begin{theorem}[Regular Value Theorem]\label{T-6-1}
Let $f: M\to N$ be a $C^r$ ($r\geq 1$) map and let $y\in N$ be a regular value of $f$. Then $f^{-1}(y)$ is a closed $C^r$ submanifold of $M$.
If $y\in im(f)$, then
\[
\dim\big(f^{-1}(y)\big)\,=\,m-n.
\]
\end{theorem}

\begin{theorem}[Sard's Theorem]\label{T-6-2}
Let $f:M\to N$ be a $C^r$ map with
\[
r\,\geq\,\max\{1,m-n+1\}.
\]
Then $f(C_f)$ has zero measure in $N$.
\end{theorem}

We now suppose $M,N$ are oriented manifolds of equal dimensions. Let
$\Lambda$ be a relatively compact open subset of $M$. That means $\Lambda\subset M$ is open and has compact closure in $M$. For simplicity,  we use the notation $f\mbox{\large $\pitchfork$}_{\Lambda}y$ to indicate that $y$ is a regular value of the restriction map $f: \Lambda \to N$.

We first define the topological degree for a special class of smooth maps. In fact, we assume $f\in C^{\infty}(\Lambda,N)\cap C^{0}(\bar{\Lambda},N)$, $f\mbox{\large $\pitchfork$}_{\Lambda}y$ and $f^{-1}(y)\cap\partial \Lambda=\emptyset$ and write the topological degree of $y$ and $f$ in $\Lambda$ as $\deg(f,\Lambda,y)$. In case that $ f^{-1}(y)\cap\Lambda=\emptyset$, we define $\deg(f,\Lambda,y)=0$. In case that $f^{-1}(y)\cap\Lambda\neq\emptyset$, the Regular Value Theorem implies it consists of finite points. For each point $x\in f^{-1}(y)\cap\Lambda$, the sign $\operatorname{sgn}(f,x)=+1$, if the tangent map $df_{x}:T_{x} M \to T_{y}N$  preserves orientation; Otherwise, $\operatorname{sgn}(f,x)=-1$.

\begin{definition}
Suppose $f^{-1}(y)\cap\Lambda=\{x_1,\cdots,x_k\}$. The degree $\deg(f,\Lambda,y)$ is defined as
\[
\deg(f,\Lambda,y)\,=\,\sum\nolimits_{i=1}^k \operatorname{sgn}(f,x_i).
\]
\end{definition}

In what follows let $I$ represent the interval $[0,1]$.
\begin{proposition}\label{P-6-4}
Let  $f_{i}\in
C^{\infty}(\Lambda,N)\cap C^{0}(\bar{\Lambda},N)$ satisfy $f_{i}\mbox{\large $\pitchfork$}_{\Lambda}y
$ and $f_{i}^{-1}(y)\cap\partial \Lambda=\emptyset$ for $i=0,1$. Suppose there
exists  a homotopy
\[
H \,\in\, C^{0}({I\times \bar\Lambda},N)
\]
such that
\begin{itemize}
\item[$(i)$]$H(0,\cdot)=f_0(\cdot)$, \, $H(1,\cdot)=f_1(\cdot)$,
\item[$(ii)$] $H(I\times\partial \Lambda)\subset N\setminus\{y\}$.
 \end{itemize}
 Then
\[
\deg(f_{0},\Lambda,y)\,=\, \deg(f_{1},\Lambda,y).
\]
\end{proposition}

The lemma below is a consequence of Sard's Theorem.
\begin{lemma}\label{L-6-5}
For any $f\in C^{0}(\bar{\Lambda},N)$ and $y\in N$, if $f^{-1}(y)\cap\partial \Lambda=\emptyset$, then there exist
\[
g\,\in\, C^{\infty}(\Lambda,N)\cap C^{0}(\bar{\Lambda},N)
\quad \text{and}\quad H\,\in\, C^{0}(I\times\bar{\Lambda},N)
\]
satisfying the following properties:
\begin{itemize}
\item[$(i)$] $g\mbox{\large $\pitchfork$}_{\Lambda} y$;
\item[$(ii)$] $H(0,\cdot)=f(\cdot)$,\, $H(1,\cdot)=g(\cdot)$;
\item[$(iii)$] $H(I\times\partial \Lambda)\subset N\setminus\{y\}$.
\end{itemize}
\end{lemma}

We are ready to define the topological degree for general continuous maps.
\begin{definition}
Let $f\in C^{0}(\bar{\Lambda},N)$ and $y\in N$ satisfy $f^{-1}(y)\cap\partial \Lambda=\emptyset$. We define
\[
\deg(f,\Lambda,y)\,=\, \deg(g,\Lambda,y),
\]
where $g$ is given by Lemma~\ref{L-6-5}.
\end{definition}

Because of Proposition~\ref{P-6-4}, $\deg(g,\Lambda,y)$ is well-defined and does not depend on the particular choice of $g$.

\begin{theorem}
Let $f_{i}\in C^{0}(\bar{\Lambda},N)$ satisfy
$f_{i}^{-1}(y)\cap\partial \Lambda=\emptyset$ for $i=0,1$. If there exists a homotopy
\[
H\,\in\, C^{0}(I\times\bar{\Lambda},N)
\]
such that
\begin{itemize}
\item[$(i)$] $H(0,\cdot)=f_{0}(\cdot)$,\, $H(1,\cdot)=f_{1}(\cdot)$,
\item[$(ii)$] $H(I\times\partial \Lambda)\subset N\setminus \{y\}$,
\end{itemize}
then
\[
\deg(f_{0},\Lambda,y)\,=\, \deg(f_{1},\Lambda,y).
\]
\end{theorem}

\begin{theorem}\label{T-6-8}
Let $\gamma:I\to N$ be a continuous curve satisfying $f^{-1}(\gamma)\cap\partial \Lambda=\emptyset$. Then
\[
\deg\big(f,\Lambda,\gamma(1)\big)\, =\, \deg\big(f,\Lambda,\gamma(0)\big).
\]
\end{theorem}

\begin{theorem}\label{T-6-9}
If $\deg\big(f,\Lambda,y\big)\neq 0$, then  $f^{-1}(y)\cap\Lambda
\neq\emptyset$.
\end{theorem}

\bigskip
\noindent \textbf{Acknowledgements}. The author would like to thank Steven Gortler for a great many helpful conversations and encouragement. He also thanks both the Institute for Computational and Experimental Research Mathematics (ICERM) at Brown University and the organizers for hosting the workshop on circle packings and geometric rigidity that indirectly led to the development of this paper. Finally, he thanks NSFC (No.11601141 and No.11631010) for financial supports.

\bigskip
\noindent Ze Zhou, zhouze@hnu.edu.cn\\[2pt]
\emph{School of Mathematics, Hunan University, Changsha 410082, P.R. China.}
\end{document}